\documentclass[review]{elsarticle}

\usepackage{lineno,hyperref}
\modulolinenumbers[5]

\usepackage{geometry}

\usepackage{graphicx,float}
\usepackage{amssymb,amsmath,amsfonts,longtable,bm,epsfig,subfigure,color}
\usepackage{natbib} \bibpunct{(}{)}{,}{a}{,}{,}

\def\Fbar{ {\overline F}}

\def\Hbar{ {\overline H}}
\newtheorem{definition}{Definition}
\newtheorem{theorem}{Theorem}
\newtheorem{proof}{Proof}

\newtheorem{lemma}{Lemma}
\newtheorem{example}{Example}
\newtheorem{remark}{Remark}
\begin{document}
	
	\begin{frontmatter}
		
		\title{Stochastic Comparisons of Second-Order Statistics from Dependent and Heterogenous Modified Proportional  Hazard Rate Observations}
		
		%\author{Rongfang Yan, Jiandong Zhang and Junrui Wang}{\corref{mycorrespondingauthor}}
		%\address{College of Mathematics and Statistics,  Northwest Normal University, Lanzhou 730070, China}
		
		%\author{Jiale Niu}
		%\author{Jiandong Zhang} 
		\author{Jiale Niu\corref{mycorrespondingauthor}}
		\cortext[mycorrespondingauthor]{Corresponding author}
		\ead{jiale.niu@outlook.com}
		%\ead[url]{www.elsevier.com}Niubility

		\address{College of Mathematics and Statistics, Northwest Normal University, Lanzhou 730070, China}
		
		\begin{abstract}
			In this manuscript, we study stochastic comparisons of the second-order statistics from dependent or independent observations with modified proportional hazard rates models. First, we establish the usual stochastic order of the second-order statistics from dependent and heterogeneous observations. Second, sufficient conditions are provided in the hazard rate order of the second-order statistics from independent observations. Then, we investigate the hazard rate order of the second-order statistics arising from two sets of independent multiple-outlier modified proportional hazard rates observations. Finally, some numerical examples are given to illustrate the theoretical findings.
		\end{abstract}
		
		\begin{keyword}
			Second-order statistic; Archimedean copula; Majorization; Modified proportional hazard rate model; Stochastic orders.

			%\texttt{elsarticle.cls}\sep \LaTeX\sep Elsevier \sep template
			\MSC[2010]  Primary 90B25, Secondary 60E15, 60K10
		\end{keyword}

	\end{frontmatter}
	
	%\linenumbers
	\section{Introduction}
	Let $X_{1:n} \leq X_{2:n} \leq \dots \leq X_{n:n}$ denote the order statistics arising from random variables $X_1, X_2,\dots, X_n$.
	The kth order statistic $X_{k:n}$ is well known for being the lifetime of a $(n-k+1)$-out-of-$n$ system, a prominent structure of redundancy in fault-tolerant systems in reliability theory that has been extensively explored in the literature; see, for example, \cite{2006Some} and \cite{2013Comparisons}.
	%It is well known that the kth order statistic $X_{k:n}$ can be considered as the lifetime of a $(n-k+1)$-out-of-$n$ system, a very popular structure of redundancy in fault-tolerant systems in reliability theory
	%that has been studied extensively in the literature; see for example, Li and Zhao [1], and Ding et al.[2].
	A parallel system's lifetime corresponds to the largest order statistics $X_{n:n}$, while a series system's lifetime corresponds to the smallest order statistic $X_{1:n}$. Order statistics are thus equivalent to the study of $k$-out-of-$n$ systems, which is a suitable research area in reliability.
	%In particular, the lifetime of a parallel system corresponds to the largest order statistics $X_{n:n}$ and
	%the lifetime of a series system corresponds to the smallest order statistic $X_{1:n}$. Thus, the study of order
	%statistics is equivalent to that of $k$-out-of-$n$ systems, which is naturally an important research topic in
	%reliability. 
	Many fields, including reliability theory, goodness-of-fit tests, auction theory, insurance, actuarial science, operations research, and life testing, use order statistics and statistics closely linked to order statistics.
	%Order statistics and those statistics closely related to order statistics play a prominent role
	%in many areas such as reliability theory, goodness-of-fit tests, auction theory, insurance, actuarial science, operations research and life testing.
	One may refer to \cite{2004Order} , and \cite{199810,1998Order} for comprehensive discussions.
	The method of stochastic ordering is useful for measuring the magnitude and variability of random variables. It has been frequently utilized in reliability theory to compare the lifetimes of different systems, in economics to conduct capital allocation and choose the optimal portfolios, and in actuarial science to optimize risk measures or expected utility functions.
	We refer \cite{2010Some}'s excellent monograph for more details on stochastic ordering. \cite{1971Comparisons} were among the first to examine stochastic order statistics comparisons deriving from heterogeneous exponential variables. \cite{1976Stochastic}, \cite{1996Some},\cite{1997Stochastic}, \cite{Khaledi2000SOME}, \cite{2006COMPARISON}, \cite{2007Stochastic,2011On}, \cite{2013ORDERING}, \cite{2014On}, and \cite{2016Comparisons} are just a few of the many researchers who have worked on this topic since then.
	%Stochastic ordering is a useful tool for measuring the magnitude and variability among random
	%variables. It has been widely used to compare the lifetimes of different systems in reliability theory, to
	%conduct capital allocation and choose the optimal portfolios in economics, and to optimize the risk
	%measures or expected utility function in actuarial science.
	%For more details on stochastic orders, we
	%refer the reader to the excellent monograph by Shaked and Shanthikumar [6]. Pledger and Proschan
	%[7] were among the first to discuss stochastic comparisons of order statistics arising from heterogeneous exponential variables. Since then, many researchers have worked on this topic including
	%Proschan and Sethuraman [8], Kochar and Rojo [9], Dykstra et al. [10], Khaledi and Kochar [11],
	%Bon and P\v{a}lt\v{a}nea [12], Kochar and Xu [13,14], Balakrishnan and Zhao [15], Zhao and Zhang [16],
	%and Balakrishnan and Torrado [17].
	The second-order statistic, which characterizes the lifetime of the $(n-1)$-out-of-$n$ system
	(referred to as the fail-safe system in reliability theory, see Barlow and Proschan, [18]) and yields the
	winner’s price for the bid in the second-price reverse auction (see, for example, \cite{2004Mean,2005A}), 
	has attracted the attention of many researchers for studying its ordering properties. 
	
	A fail-safe system is a significant $k$-out-of-$n$ system, this fault tolerance system is widely applied to the various critical safety system.   
	In the literature, there has been an abundant study on stochastic properties of fail-safe systems. For instance, 
	\cite{Paltanea2008} established the comparison of fail-safe systems with respect to hazard rate order.
	\cite{Zhao2009} obtained the comparison results for two fail-safe systems consisting of heterogeneous and homogeneous components in terms of the likelihood ratio order. 
	\cite{Zhao2010} provided the dispersive order of fail-safe systems with heterogeneous exponential components.
	\cite{Zhao2011} given the right spread order of the fail-safe systems from heterogeneous exponential components.
	\cite{Balakrishnan2015} improved the ordering results for fail-safe systems with exponential components. 
	\cite{cai2017hazard}  compared the hazard rate functions of two fail-safe systems consisting of independent and multiple-outlier proportional hazard rate components. For more the investigations of order statistics, one can may refer to \cite{torrado2015magnitude,torrado2015stochastic,barmalzan2016likelihood,barmalzan2019ordering,das2021some,das2021orderings}.
	
	%It should be mentioned that our results established here can be also applied in the area of auction theory. For example,
	%in the second-price reverse auction, bidders (sellers) submit sealed bids to the auctioneer (buyer) who solicits the purchase of
	%items when the auction begins. The lowest bidder wins the bid and will be paid the amount of the second-lowest price from
	%the auctioneer. This type of auction is often used by large corporations and government departments to purchase supplies and services. The cost of the auctioneer can be expressed as $X_{2:n}$ if there are $n$ bidders submitting prices $X_1,\dots,X_n$.
	%It may happen that some of the bidders drop out of the auction before the beginning due to some unforeseen circumstances.
	
	In reliability theory, to model the lifetime data with different hazard shapes, it is desirable to introduce flexible families of distributions, and to this end, there are two methods have been commonly
	used to characterize lifetime distribution with considerable flexibility. One method is to adopt the well-known families of distributions, for example, Gamma, Weibull, and Log-normal, which have been
	studied quite extensively in the literature, for more discussions on this topic, we refer readers to \cite{1994Continuous,1995Continuous,1993Exponentiated}.
	\cite{MARSHALL1997A} developed a new method to introduce one parameter to a base distribution results in a new family of distribution with more flexibility. For example, for a baseline distribution function $F$ with support $\mathbb{R}^+$ and corresponding survival function $\overline{F}$, the new distribution functions can be defined as
	\begin{equation}
		G(x;\alpha)=\frac{F(x)}{1-\overline{\alpha}\overline{F}(x)},\quad x,\alpha\in \mathbb{R}^+,\quad \overline{\alpha}=1-\alpha,
	\end{equation}
	\begin{equation}
		H(x;\alpha)=\frac{\alpha F(x)}{1-\overline{\alpha}F(x)},\quad x,\alpha\in \mathbb{R}^+,\quad \overline{\alpha}=1-\alpha.
	\end{equation}
	\cite{MARSHALL1997A} originally proposed the family of distributions in (1) and studied it for the case when $F$ is a Weibull distribution. When $F$ has probability density and hazard rate functions as $f$
	and $r_F$, respectively, then the hazard rate function of $G$ is given by
	\begin{equation*}
		r_F(x;\alpha)=\frac{1}{1-\overline{\alpha}\overline{F}(x)}r_F(x), \quad x,\alpha\in \mathbb{R}^+,\quad \overline{\alpha}=1-\alpha.
	\end{equation*}
	Therefore, one can observe that if $r_F(x)$ is decreasing (increasing) in x, then for $0<\alpha\leq1(\alpha\geq1)$,
	$r_F(x;\alpha)$ is also decreasing in x. Moreover, one can observe that $r_F(x)\leq r_F(x;\alpha)$ for $0<\alpha\leq1$, and $r_F(x;\alpha)\leq r_F(x)$ for $\alpha\geq1$. For this reason, the parameter $\alpha$ in (1) is referred to as a tilt parameter (see \cite{2007Life}). Note that (1) is equivalent to (2) if $\alpha$ in (1) is changed to $\frac{1}{\alpha}$. The proportional hazard rates (PHR) have important applications in reliability and survival analysis. The random variables $X_1,\cdots, X_n$ are said to follow: PHR model if $X_i$ has the survival function $\overline{F}_{X_i}(x)=\overline{F}^{\lambda_i}(x),i=1,\dots,n$ where $\overline{F}$ is the baseline survival function and $(\lambda_1,\dots, \lambda_n)$ is the frailty vector. It is well-known that the Exponential, Weibull, Lomax and Pareto distributions are special cases of the PHR model. Balakrishnan et al. \cite{2017Modified} introduced two new statistical models by adding a parameter to PHR models, which are regarded as the baseline distributions in $G(x;\alpha)$. The new model is referred to as the modified proportional hazard rates (MPHR). It is given by
	\begin{equation}
		G(x;\alpha,\lambda)=\frac{1-\overline{F}^\lambda(x)}{1-\overline{\alpha}\overline{F^\lambda}(x)},\quad x,\alpha\in \mathbb{R}^+,\quad \overline{\alpha}=1-\alpha,
	\end{equation}
	%\begin{equation}
	%	H(x;\alpha,\beta)=\frac{\alpha F^\beta(x)}{1-\overline{\alpha}F^\beta(x)},\quad x,\alpha\in \mathbb{R}^+,\quad \overline{\alpha}=1-\alpha.
	%\end{equation}
	where $\lambda$ is the proportional hazard rate parameters. We denote $X\sim MPHR(\alpha,\lambda;\overline{F})$ if $X$ has the distribution
	functions $G(x;\alpha,\lambda)$. For the case $\lambda=1$, (3)  simply reduce
	to (1). For the case $\alpha=1$, (3) simply reduce to the PHR model. According to the Theorem 2.1 of Navarro et al. [14], (3) can
	be rewritten the distorted distribution of $h$, where
	\begin{equation}
		h(u;\alpha,\lambda)=\frac{1-(u)^\lambda}{1-\overline{\alpha}(u)^\lambda},\quad u=\overline{F}(x),\quad x,\alpha\in\mathbb{R}^+,\quad \overline{\alpha}=1-\alpha,
	\end{equation}
	if $\lambda=1$, (4) just as the distorted distributions of (1). For some special
	models, please refer to multiple-outlier models (\cite{2016Comparisons,2017Comparisons}), extended exponential and extend Weibull
	distribution (\cite{MARSHALL1997A}, \cite{0Stochastic}), extended Pareto distribution ( \cite{2006Marshall}) and extended Lomax distribution (\cite{2007Marshall}).
	
	However, to the best of our knowledge, there is no study on the ordering properties of the second-order statistics from the heterogeneous and dependent MPHR model. Our purpose in this paper is to investigate the existence of stochastic orderings between the second-order statistics. The rest of the paper is laid out as follows.
	
	The remaining part of the paper is organized as follows: Section 2 recalls some basic concepts
	and notations that will be used in the sequel. Section 3 obtains the usual stochastic order between the second-order statistics under the assumption that the observations are dependent on Archimedean copula. Section 4
	deals with the hazard rate order under the assumption that the observations are independent. Section 5 concludes this paper and gives some remarks.

	\section{Preliminaries}\label{pre}
	In this section, we recall some pertinent definitions and lemmas in the sequel. Throughout, the term ``increasing'' and ``decreasing'' are used in a non-strict  sense. $ h^{-1} (u)$ denotes the inverse function of $h(u)$, and  $h^{'}(u)$ denotes the derivative function of $h(u)$ when they appear. Let $\mathbb R=(-\infty, +\infty)$, $\mathbb R_+=[0, +\infty)$, $\mathbb R_{++}=(0, +\infty)$,  $\mathbb N=\{1,2, \dots, n\}$, $\mathcal D =\{\bm x: x_1\ge x_2 \ge \dots\ge x_n\}$ and $\mathcal I =\{\bm x: x_1\leq x_2 \leq \dots\leq x_n\}$. %For $1\le i <j\le n$, we use $ \bm x^{\{ij\}}=\{x_1,x_2, \dots,  x_n\}\backslash \{x_i, x_j\}$ to denote the sub-vector of  $\bm x$ with the $i$th and the $j$th component entries deleted.    
	
	Stochastic order is a very useful tool to compare random variables arising from reliability theory, operations research, actuarial science,  economics, finance, and so on. Let  $X$ and $Y$  be two random variables with distribution functions $F(t)$ and $G(t)$,
	survival functions $\overline  F(t) =1- F(t)$ and $\overline  G(t) =1- G(t)$, probability density functions $f(t)$  and $g(t)$,  hazard rate functions $h_X(t) = f(t) /\overline  F(t) $ and $h_Y(t) = g(t)/\overline  G(t)$, and reversed hazard rate functions $\tilde r_X(t) = f(t) /  F(t)  $ and $\tilde r_Y(t) = g(t)/ G(t)$, respectively. 
	\begin{definition}
		A random variable	 $X$ is said to be smaller than $Y$ in the
		\begin{enumerate} [\rm (i)]
			\item  usual stochastic order (denoted by $X \le_{\rm st} Y$) if  $\overline F(t) \le \overline G(t)$, for $t \in \mathbb R;$
			\item  { hazard rate order} (denoted by $X\leq_{\rm hr}Y$) if
			$h_{X}(t)\ge h_{Y}(t)$ for all $t\in\mathbb{R}$, or  equivalently, if $\overline {G}(t)/\overline {F}(t)$ is increasing in $t\in\mathbb{R};$
			\item  { reversed hazard rate order} (denoted by $X\leq_{\rm rh}Y$) if
			$\tilde r_{X}(t)\le \tilde r_{Y}(t)$ for all $t\in\mathbb{R}$, or equivalently, if $  {G}(t)/  {F}(t)$ is increasing in $t\in\mathbb{R}$.
		\end{enumerate}	
	\end{definition}
	It is well known that the (reversed) hazard rate order implies the usual stochastic order, but the reversed statement is not true in general. For more comprehensive discussions on various stochastic orders and their applications, one may refer to the classical monograph \cite{Shaked2007} and \cite{belzunce2015introduction}.

	Majorization order is an useful tool in establishing various inequalities arising from many research areas. For any  two real-valued  vectors $\bm{x}=(x_{1}, x_{2},\ldots, x_{n})$ and $\bm{y}=(y_{1}, y_{2},\ldots, y_{n})$, let $x_{1:n}\leq x_{2:n}\leq \dots\leq x_{n:n}$ and $y_{1:n}\leq y_{2:n}\leq \dots\leq y_{n:n}$ be their increasing arrangements, respectively. 
	\begin{definition}
		A vector $\bm{x} $ is said to 
		\begin{enumerate}[{\rm (i)}]
			\item   majorized $\bm{y} $ {\rm (denoted by $\bm{x}\overset{\rm{m}}{\succeq}\bm{y}$)},
			if 
			$\sum_{i=1}^{j}x_{i:n}\leq\sum_{i=1}^{j}y_{i:n}$ for every $j=1,2,\ldots,n-1$, and $\sum_{i=1}^{n}x_{i:n}=\sum_{i=1}^{n}y_{i:n}${\rm;}	
			\item   weakly supermajorized $\bm{y} $ {\rm (denoted by $\bm{x}\stackrel{{\rm w}}{\succeq}\bm{y}$)}, if $\sum^{j}_{i=1}x_{i:n}\leq \sum^{j}_{i=1}y_{i:n}$ for every $j=1,\ldots, n${\rm;}	
			\item   weakly submajorized $\bm{y} $ {\rm (denoted by $\bm{x}\underset{{\rm w}}{\succeq}\bm{y}$)}, if $\sum^{n}_{i=j}x_{i:n}\geq \sum^{n}_{i=j}y_{i:n}$ for every $j=1,\ldots, n${\rm;}
			%\item   reciprocal majorized $\bm{y}$ {\rm (denoted by $\bm{x}\stackrel{{\rm rm}}{\succeq}\bm{y}$)}, if $\sum^{j}_{i=1}\frac{1}{x_{i:n}}\leq \sum^{j}_{i=1}\frac{1}{y_{i:n}}$ for every $j=1,\ldots, n${\rm.}
		\end{enumerate}	
		%	A vector $\bm{x}\in\mathbb{R}^{n}$ is said to majorize  $\bm{y}\in\mathbb{R}^{n}$, denoted by $\bm{x}\overset{\rm{m}}{\succeq}\bm{y}$,
		%	if
		%	$\sum_{i=1}^{j}x_{i:n}\leq\sum_{i=1}^{j}y_{i:n}$, for all $j=1,2,\ldots,n-1$, and $\sum_{i=1}^{n}x_{i:n}=\sum_{i=1}^{n}y_{i:n}$.
	\end{definition}
	
	%It is clear that $\bm{x}\overset{\rm{m}}{\succeq}\bm{y}$ implies $\bm{x}\overset{\rm{w}}{\succeq}\bm{y}$, while the reverse is not true in general.
	It is known that  
	\begin{equation*}
		\bm{x}\underset{{\rm w}}{\succeq}\bm{y}\Longleftarrow\bm{x}\overset{\rm{m}}{\succeq}\bm{y}\Longrightarrow\bm{x}\overset{\rm{w}}{\succeq}\bm{y}\Longrightarrow\bm{x}\stackrel{{\rm rm}}{\succeq}\bm{y}
	\end{equation*}
	for any two non-negative vectors   $\bm{x} $   and   $ \bm{y}$, while the reverse is not true in general. The notion
	of majorization is quite useful in establishing various inequalities.  For more detailed discussion on the theory of majorization and its  applications, one may refer to \cite{Marshall2011} and \cite{Balakrishnan2013}. 
	
	The following lemma is very helpful to establish
	inequalities for asymmetrical multivariate functions in terms of weak majorization order.
	
	%\textcolor{red}{***Q: Add the detailed pages/labels in the following lemmas!***}
	
	%\begin{definition}{\rm\citep[Definition I: 3.A.1 on Page 80 of][]{Marshall2011}}\label{concavedefition}
	%	A real-valued function $\phi$, defined on a set $A\subset\mathbb R^n$, is said to be Schur-convex $[$Schur-concave$]$ on  $A$ if and only if  $\bm{x}\overset{\rm{m}}{\succeq}\bm{y}$ implies $\phi(\bm{x}) \ge[\le] \phi(\bm{y}) $, for any $\bm{x, y}\in\mathbb{R}^n$.  
	%	%Meanwhile, $\phi$ is said to be Schur-conacve function on $A$ if  $-\phi$ is Schur-convex on  $A$.
	%\end{definition}
	\begin{lemma}{\rm (\cite{Marshall2011})}\label{wconvex}
		Let $\phi$ be a real valued function defined on  $\mathcal D_n$ and continuously differential on the interior of $\mathcal D_n$. Denote the partial derivative of $\phi$ with respect to its $k$-th argument by $\phi_{(k)}(\bm z)=\partial \phi(\bm z) /\partial z_k$, $k=1,\dots, n$. Then,  $\phi (\bm x) \le \phi (\bm y)$  whenever
		$ \bm{x}\overset{{\rm w}}{\preceq}\bm{y}  \text{ on } \mathcal D_n$ if and only if $  0 \ge \phi_{(1)}(\bm z) \ge \phi_{(2)}(\bm z)  \ge \cdots \ge \phi_{(n)}(\bm z).$ Similarly, $\phi (\bm x) \le \phi (\bm y)$  whenever
		$ \bm{x}{\preceq}_{{\rm w}}\bm{y}  \text{ on } \mathcal D_n$ if and only if $ \phi_{(1)}(\bm z) \ge \phi_{(2)}(\bm z)  \ge \cdots \ge \phi_{(n)}(\bm z) \ge 0.$
	\end{lemma}
	\begin{lemma}\label{kundu}
		\cite{2016Some}Let $h:\mathcal{I_+}\rightarrow\mathbb{R}$ be a function, continuously differentiable on the interior of $\mathcal{I_+}.$ Then, for ${\bm x}$,${\bm y} \in \mathcal{I_+}$, ${\bm x}\overset{\rm m}{\succeq}{\bm y}$ implies $h({\bm x})\geq(respectively, \leq) h({\bm y})$ if and only if, $h_{(k)}({\bm z})$ is increasing (respectively, decreasing) in $k=1,\dots,n,$ where $h_{(k)}({\bm z})=\partial h_({\bm z})/\partial z_k$ denotes the partial derivative of $h$ with respect to its $k$th argumemt.
	\end{lemma}
	\begin{lemma}\label{proof}
		Let the function $T(x)=\frac{x^2}{(1-p+px)^2}$, where $p\in(0,1]$. Then $T(x)$ is increasing in $x\in [0,1]$.
	\end{lemma}
	\begin{proof}
		Taking the derivative of $T(x)$ with respect to $X\in (0,1]$, we immediately get
		\begin{eqnarray*}
			T^\prime(x)&\overset{\rm sgn}{=}&2x(1-p+px)[1-p+px-px]\\
			&=&2x(1-p+px)[1-p]\geq0.
		\end{eqnarray*}
		Then,we get the required result.
	\end{proof}
	
	Next, let us review the concept of copula. For a random vector $\bm{X}=(X_1,X_2,\ldots,X_n)$ with the joint distribution function $H$ and respective marginal distribution functions $F_1,F_2, \ldots,F_n$,  the \emph{copula} of $X_1,X_2,\ldots,X_n$ is a distribution function $C:[0,1]^n\mapsto [0,1]$, satisfying $$H(\bm{x})=\mathbb{P}(X_1\le x_1,X_2\le x_2,\ldots,X_n\le x_n)=C(F_1(x_1),F_2(x_2),\ldots,F_{n}(x_n)).$$  Similarly, a \emph{survival copula} of $X_1,X_2,\ldots,X_n$  is a survival function $\hat{C} :[0,1]^n\mapsto [0,1]$, satisfying
	\begin{equation*}
		\Hbar(\bm{x})=\mathbb{P}(X_1>x_1,X_2>x_2,\ldots,X_n>x_n)=\hat{C}(\Fbar_1(x_1), \Fbar_2(x_2),\ldots,\Fbar_n(x_n)),
	\end{equation*}
	where $ \Hbar(\bm{x})$ is the joint survival function.

	Notably, since  Archimedean copula  has well the mathematical properties,
	which is applied to the various areas. In this paper, we will  sometimes employ Archimedean copulas to model the dependence structure among components lifetimes.
	\begin{definition}{\rm\citep[Expression 4.6.1 on Page 151 of][]{Nelsen2006}}
		For a decreasing and continuous function $\phi:[0,1]\mapsto [0,+\infty]$ such that $\phi(0)=+\infty $ and $\phi(1)=0$,  and  let $\psi=\phi^{-1}$ be the pseudo-inverse of $\phi$. Then
		$$C_\phi (u_1, u_2,...,u_n)=\psi \Big(\sum_{i=1}^n\phi (u_i)\Big),  \text{ for all } u_i\in [0,1], \quad i=1,2,\dots,n $$
		is said to be an Archimedean copula with generator $\phi$ if $(-1)^k\phi^{(k)}(x)\ge 0$ for $k=0,1, \dots,n-2$ and $(-1)^{n-2}\phi^{(n-2)}(x)$ is decreasing and convex.
	\end{definition}
	
	%It is a partial order as not every pair of copulas is concordance-comparable. Besides, the canonical scale-free dependence measures given by Kendall's tau and Spearman's rho are well known to be increasing with respect to the concordance order.

	\section{Usual stochastic order of  second-order statistics from dependent and heterogeneous observations}\label{st}
	In this section, we carry out the usual stochastic order of second order statistics from dependent and heterogeneous MPHR samples. Let $\bm{X}=(X_1,X_2,\dots,X_n)$ and $\bm{Y}=(Y_1,Y_2,\dots,Y_n)$ be two n-dimensional vectors of heterogeneous dependent random observations, with $\bm{X}\thicksim MPHR(\bm{\alpha},\bm{\lambda};\overline{F},\psi)$ and $\bm{Y}\thicksim MPHR(\bm{\beta},\bm{\mu};\overline{F},\psi)$, where $\bm{\alpha}=(\alpha_1,\alpha_2,\dots,\alpha_n),\ \bm{\lambda}=(\lambda_1,\lambda_2,\dots,\lambda_n),\ \bm{\beta}=(\beta_1,\beta_2,\dots,\beta_n),$  and $\bm{\mu}=(\mu_1,\mu_2,\dots,\mu_n)$. Here, $X_i\thicksim MPHR(\alpha_i,\lambda_i;\overline{F},\psi)$ and $Y_i\thicksim MPHR(\beta_i,\mu_i;\overline{F},\psi)$, for $i=1,2,\dots,n$. We recall that $\overline{F}$ is the baseline survival function, $\psi$ is generator of the associated Archimedean survival copula. 
	
	Next, we establish sufficient conditions for the usual stochastic order, whenever the modified proportional hazard rate parameters may be different with the tilt parameters being equal.
	\begin{theorem}\label{tp}
		Let $X_1,X_2\dots,X_n$ be n statistically dependent heterogeneous random variables with $X_i\thicksim MPHR(\alpha,\lambda_i;\overline{F};\psi)$, where $0<\alpha\leq1,$ and let $N_1$ be a non-negative integer-valued random variable independent of $ X_{i}'$s. Let $Y_1,Y_2\dots,Y_n$ be another n statistically dependent heterogeneous random variables with $Y_i\thicksim MPHR(\alpha,\mu_i;\overline{F};\psi)$, and let $N_2$ be another non-negative integer-valued random variable independent of $ Y_{i}'$s. Suppose $\lambda,\mu\in\mathcal D_+$ $(         or
		\ \mathcal{I_+})$, if $N_1\geq_{\rm st}N_2,$ and $\psi$ is log-concave. Then we have
		
		\begin{equation*}
			\lambda\overset{\rm w}{\succeq}\mu\Rightarrow X_{2:N_1}\geq_{\rm st}Y_{2:N_2}
		\end{equation*}
		
		%	$X_{1},X_{2},\dots,X_{n}$ are  non-negative random variables with $X_{i}\thicksim \overline F_{X_{i}}(x)$ and survival copula $\hat C$, where $\overline F_{X_{i}}(x)$ is the survival function of $X_{i}$,  for $i=1,2, \dots, n$. Let $I_{p_1}, I_{p_2}, \dots, I_{p_n}$$[I_{p_1^*}, I_{p_2^*}, \dots, I_{p_n^*}]$   be   independent Bernoulli random variables, independent of $ X_{i}'$s, with $\mathbb E[I_{p_i}]=p_i$$[\mathbb E[I_{p_i^*}]=p_i^*]$,  for $i=1,2, \dots, n$. Assume that
		%	\begin{enumerate}[\rm(i)]
			%	\item $h:[0,1]\to \mathbb R_{++} $   is differentiable and strictly decreasing;
			%	\item $ph'(p)$ is increasing in $p\in[0,1];$ 
			%	\item  $X_{1} \le_{\rm st}X_{2}\le_{\rm st} \dots \le_{\rm st}X_{n}$. Then,
			%	\end{enumerate}
		%	For $\bm {h(p)} \in \mathcal D $ and $\bm {h(p^*)} \in \mathcal D $,   we have
		%	\begin{equation*}
			%	(h(p_1), h(p_2),\dots, h(p_n)) \overset{\rm w}{\succeq}   (h(p_1^*), h(p_2^*),\dots, h(p_n^*) )  
			%	\Rightarrow   Y_{2:n} \ge_{\rm st}Y_{2:n}^*.
			%	\end{equation*} 
	\end{theorem}
	
	\begin{proof}
		We present the proof when  $\lambda,\mu\in\mathcal D_+$. The proof is similar when these vectors belong to $\mathcal I_+$. The survival function of $X_{2:n}$ can be express as 
		\begin{eqnarray*}
			\overline{F}_{X_{2:n}}(x)=\sum\limits_{i=1}^{n}\psi\left( \sum\limits_{j\neq i}^{n}\phi\left( \frac{\alpha\overline{F}^{\lambda_j}(x)}{1-\overline{\alpha}\overline{F}^{\lambda_j}(x)}\right) \right) -(n-1)\psi\left(\sum\limits_{i=1}^{n}\phi\left(  \frac{\alpha\overline{F}^{\lambda_i}(x)}{1-\overline{\alpha}\overline{F}^{\lambda_i}(x)}\right)  \right) 
		\end{eqnarray*}
		If $N_1\geq_{\rm st}N_2$, then
		\begin{eqnarray*}
			\overline{F}_{X_{2:N_1}}(x)&=&\sum\limits_{m=1}^{n}P(X_{2:N_1}> x|N_1=m)\mathbb{P}(N_1=m)\\
			&=&\sum\limits_{m=1}^{n}P(X_{2:m}> x)\mathbb{P}(N_1=m)\\
			&\geq&\sum\limits_{m=1}^{n}P(X_{2:m}> x)\mathbb{P}(N_2=m).
		\end{eqnarray*}
		To obtain the desired result, it is then sufficient to show that $\overline{F}_{X_{2:m}}(x)\geq\overline{F}_{Y_{2:m}}(x),\ m=1,\dots,n. $ For any $k= {1,2,\dots,m}$, taking the partial derivative of $\overline{F}_{X_{2:m}}(x)$ with respect to $\lambda_k,$ we have 
		\begin{eqnarray*}
			\frac{\partial\overline{F}_{X_{2:m}}(x)}{\partial \lambda_k}
			&=& \phi^\prime\left( \frac{\alpha\overline{F}^{\lambda_k}(x)}{1-\overline{\alpha}\overline{F}^{\lambda_k}(x)}\right)\frac{\alpha {\rm log}{\overline{F}(x)}\overline{F}^{\lambda_k}(x)}{[1-\overline{\alpha}\overline{F}^{\lambda_k}(x)]^2}\\
			&&\times\left[ \sum\limits_{i\neq k}^{m}\psi^\prime\left( \sum\limits_{j\neq i}^{m}\phi\left( \frac{\alpha\overline{F}^{\lambda_j}(x)}{1-\overline{\alpha}\overline{F}^{\lambda_j}(x)}\right) \right) -(m-1)\psi^\prime\left(\sum\limits_{i=1}^{m}\phi\left(  \frac{\alpha\overline{F}^{\lambda_i}(x)}{1-\overline{\alpha}\overline{F}^{\lambda_i}(x)}\right)  \right)\right] \\
			&\leq&0
		\end{eqnarray*}
		where the inequality is due to the fact that the decreasing and convex properties of $\phi$.
		Hence, we have 
		\begin{eqnarray*}
			&&\frac{\partial\overline{F}_{X_{2:m}}(x)}{\partial \lambda_k}-\frac{\partial\overline{F}_{X_{2:m}}(x)}{\partial \lambda_l}\\
			&=&\phi^\prime\left( \frac{\alpha\overline{F}^{\lambda_k}(x)}{1-\overline{\alpha}\overline{F}^{\lambda_k}(x)}\right)\frac{\alpha {\rm log}{\overline{F}(x)}\overline{F}^{\lambda_k}(x)}{[1-\overline{\alpha}\overline{F}^{\lambda_k}(x)]^2}\\
			&&\times\left[ \sum\limits_{i\neq k}^{m}\psi^\prime\left( \sum\limits_{j\neq i}^{m}\phi\left( \frac{\alpha\overline{F}^{\lambda_j}(x)}{1-\overline{\alpha}\overline{F}^{\lambda_j}(x)}\right) \right) -(m-1)\psi^\prime\left(\sum\limits_{i=1}^{m}\phi\left(  \frac{\alpha\overline{F}^{\lambda_i}(x)}{1-\overline{\alpha}\overline{F}^{\lambda_i}(x)}\right)  \right)\right] \\
			&&-\phi^\prime\left( \frac{\alpha\overline{F}^{\lambda_l}(x)}{1-\overline{\alpha}\overline{F}^{\lambda_l}(x)}\right)\frac{\alpha {\rm log}{\overline{F}(x)}\overline{F}^{\lambda_l}(x)}{[1-\overline{\alpha}\overline{F}^{\lambda_l}(x)]^2}\\
			&&\times\left[ \sum\limits_{i\neq l}^{m}\psi^\prime\left( \sum\limits_{j\neq i}^{m}\phi\left( \frac{\alpha\overline{F}^{\lambda_j}(x)}{1-\overline{\alpha}\overline{F}^{\lambda_j}(x)}\right) \right) -(m-1)\psi^\prime\left(\sum\limits_{i=1}^{m}\phi\left(  \frac{\alpha\overline{F}^{\lambda_i}(x)}{1-\overline{\alpha}\overline{F}^{\lambda_i}(x)}\right)  \right)\right] \\
			&\overset{\rm sgn}{=}&\frac{\psi\left( \phi\left( \frac{\alpha\overline{F}^{\lambda_l}(x)}{1-\overline{\alpha}\overline{F}^{\lambda_l}(x)}\right)\right) }{\psi^\prime\left( \phi\left( \frac{\alpha\overline{F}^{\lambda_l}(x)}{1-\overline{\alpha}\overline{F}^{\lambda_l}(x)}\right)\right)}\frac{1}{1-\overline{\alpha}\overline{F}^{\lambda_l}(x)}\\
			&&\times\left[ \sum\limits_{i\neq l}^{m}\psi^\prime\left( \sum\limits_{j\neq i}^{m}\phi\left( \frac{\alpha\overline{F}^{\lambda_j}(x)}{1-\overline{\alpha}\overline{F}^{\lambda_j}(x)}\right) \right) -(m-1)\psi^\prime\left(\sum\limits_{i=1}^{m}\phi\left(  \frac{\alpha\overline{F}^{\lambda_i}(x)}{1-\overline{\alpha}\overline{F}^{\lambda_i}(x)}\right)  \right)\right]\\
			&&-\frac{\psi\left( \phi\left( \frac{\alpha\overline{F}^{\lambda_k}(x)}{1-\overline{\alpha}\overline{F}^{\lambda_k}(x)}\right)\right) }{\psi^\prime\left( \phi\left( \frac{\alpha\overline{F}^{\lambda_k}(x)}{1-\overline{\alpha}\overline{F}^{\lambda_k}(x)}\right)\right)}\frac{1}{1-\overline{\alpha}\overline{F}^{\lambda_k}(x)}\\
			&&\times\left[ \sum\limits_{i\neq k}^{m}\psi^\prime\left( \sum\limits_{j\neq i}^{m}\phi\left( \frac{\alpha\overline{F}^{\lambda_j}(x)}{1-\overline{\alpha}\overline{F}^{\lambda_j}(x)}\right) \right) -(m-1)\psi^\prime\left(\sum\limits_{i=1}^{m}\phi\left(  \frac{\alpha\overline{F}^{\lambda_i}(x)}{1-\overline{\alpha}\overline{F}^{\lambda_i}(x)}\right)  \right)\right]\\
			&=:&P_1Q_1-U_1V_1=P_1(Q_1-V_1)+(P_1-U_1)V_1.
		\end{eqnarray*}		
		For any $1\leq k<l\leq m$, $\lambda_k\geq \lambda_l$, by the log-concave and decreasing properties of $\psi$, we have
		\begin{eqnarray*}
			\frac{\psi\left( \phi\left( \frac{\alpha\overline{F}^{\lambda_k}(x)}{1-\overline{\alpha}\overline{F}^{\lambda_k}(x)}\right)\right) }{\psi^\prime\left( \phi\left( \frac{\alpha\overline{F}^{\lambda_k}(x)}{1-\overline{\alpha}\overline{F}^{\lambda_k}(x)}\right)\right)}\geq\frac{\psi\left( \phi\left( \frac{\alpha\overline{F}^{\lambda_l}(x)}{1-\overline{\alpha}\overline{F}^{\lambda_l}(x)}\right)\right) }{\psi^\prime\left( \phi\left( \frac{\alpha\overline{F}^{\lambda_l}(x)}{1-\overline{\alpha}\overline{F}^{\lambda_l}(x)}\right)\right)}.	
		\end{eqnarray*}
		It holds that
		\begin{eqnarray*}
			(P_1-U_1)V_1&=&\left[ \frac{\psi\left( \phi\left( \frac{\alpha\overline{F}^{\lambda_l}(x)}{1-\overline{\alpha}\overline{F}^{\lambda_l}(x)}\right)\right) }{\psi^\prime\left( \phi\left( \frac{\alpha\overline{F}^{\lambda_l}(x)}{1-\overline{\alpha}\overline{F}^{\lambda_l}(x)}\right)\right)}\frac{1}{1-\overline{\alpha}\overline{F}^{\lambda_l}(x)}-\frac{\psi\left( \phi\left( \frac{\alpha\overline{F}^{\lambda_k}(x)}{1-\overline{\alpha}\overline{F}^{\lambda_k}(x)}\right)\right) }{\psi^\prime\left( \phi\left( \frac{\alpha\overline{F}^{\lambda_k}(x)}{1-\overline{\alpha}\overline{F}^{\lambda_k}(x)}\right)\right)}\frac{1}{1-\overline{\alpha}\overline{F}^{\lambda_k}(x)}\right] \\
			&&\times\left[ \sum\limits_{i\neq k}^{m}\psi^\prime\left( \sum\limits_{j\neq i}^{m}\phi\left( \frac{\alpha\overline{F}^{\lambda_j}(x)}{1-\overline{\alpha}\overline{F}^{\lambda_j}(x)}\right) \right) -(m-1)\psi^\prime\left(\sum\limits_{i=1}^{m}\phi\left(  \frac{\alpha\overline{F}^{\lambda_i}(x)}{1-\overline{\alpha}\overline{F}^{\lambda_i}(x)}\right)  \right)\right]\\
			&\overset{\rm sgn}{=}&\frac{\psi\left( \phi\left( \frac{\alpha\overline{F}^{\lambda_k}(x)}{1-\overline{\alpha}\overline{F}^{\lambda_k}(x)}\right)\right) }{\psi^\prime\left( \phi\left( \frac{\alpha\overline{F}^{\lambda_k}(x)}{1-\overline{\alpha}\overline{F}^{\lambda_k}(x)}\right)\right)}\left(\frac{1}{1-\overline{\alpha}\overline{F}^{\lambda_k}(x)}-\frac{1}{1-\overline{\alpha}\overline{F}^{\lambda_l}(x)} \right)\\
			&&+\frac{1}{1-\overline{\alpha}\overline{F}^{\lambda_l}(x)} \left(\frac{\psi\left( \phi\left( \frac{\alpha\overline{F}^{\lambda_k}(x)}{1-\overline{\alpha}\overline{F}^{\lambda_k}(x)}\right)\right) }{\psi^\prime\left( \phi\left( \frac{\alpha\overline{F}^{\lambda_k}(x)}{1-\overline{\alpha}\overline{F}^{\lambda_k}(x)}\right)\right)}-\frac{\psi\left( \phi\left( \frac{\alpha\overline{F}^{\lambda_l}(x)}{1-\overline{\alpha}\overline{F}^{\lambda_l}(x)}\right)\right) }{\psi^\prime\left( \phi\left( \frac{\alpha\overline{F}^{\lambda_l}(x)}{1-\overline{\alpha}\overline{F}^{\lambda_l}(x)}\right)\right)} \right)\\
			&\geq&0 
		\end{eqnarray*}
		For any $\lambda_k\geq\lambda_l$, owing to $\phi$ is decreasing and convex,we have 
		\begin{eqnarray*}
			\psi^\prime\left( \sum\limits_{j\neq k}^{m}\phi\left( \frac{\alpha\overline{F}^{\lambda_j}(x)}{1-\overline{\alpha}\overline{F}^{\lambda_j}(x)}\right) \right)\leq\psi^\prime\left( \sum\limits_{j\neq l}^{m}\phi\left( \frac{\alpha\overline{F}^{\lambda_j}(x)}{1-\overline{\alpha}\overline{F}^{\lambda_j}(x)}\right) \right)
		\end{eqnarray*}
		As a result, we have
		\begin{eqnarray*}
			P_1(Q_1-V_1)
			&=&	\frac{\psi\left( \phi\left( \frac{\alpha\overline{F}^{\lambda_l}(x)}{1-\overline{\alpha}\overline{F}^{\lambda_l}(x)}\right)\right) }{\psi^\prime\left( \phi\left( \frac{\alpha\overline{F}^{\lambda_l}(x)}{1-\overline{\alpha}\overline{F}^{\lambda_l}(x)}\right)\right)}\frac{1}{1-\overline{\alpha}\overline{F}^{\lambda_l}(x)}\\
			&&\times\left[ \sum\limits_{i\neq l}^{m}\psi^\prime\left( \sum\limits_{j\neq i}^{m}\phi\left( \frac{\alpha\overline{F}^{\lambda_j}(x)}{1-\overline{\alpha}\overline{F}^{\lambda_j}(x)}\right) \right)- \sum\limits_{i\neq k}^{m}\psi^\prime\left( \sum\limits_{j\neq i}^{m}\phi\left( \frac{\alpha\overline{F}^{\lambda_j}(x)}{1-\overline{\alpha}\overline{F}^{\lambda_j}(x)}\right) \right) \right]\\ 
			&\overset{\rm sgn}{=}&\sum\limits_{i\neq k}^{m}\psi^\prime\left( \sum\limits_{j\neq i}^{m}\phi\left( \frac{\alpha\overline{F}^{\lambda_j}(x)}{1-\overline{\alpha}\overline{F}^{\lambda_j}(x)}\right) \right)-\sum\limits_{i\neq l}^{m}\psi^\prime\left( \sum\limits_{j\neq i}^{m}\phi\left( \frac{\alpha\overline{F}^{\lambda_j}(x)}{1-\overline{\alpha}\overline{F}^{\lambda_j}(x)}\right) \right)\\
			&=&\psi^\prime\left( \sum\limits_{j\neq l}^{m}\phi\left( \frac{\alpha\overline{F}^{\lambda_j}(x)}{1-\overline{\alpha}\overline{F}^{\lambda_j}(x)}\right) \right)-\psi^\prime\left( \sum\limits_{j\neq k}^{m}\phi\left( \frac{\alpha\overline{F}^{\lambda_j}(x)}{1-\overline{\alpha}\overline{F}^{\lambda_j}(x)}\right) \right)\\
			&\geq&0
		\end{eqnarray*}
		
		Therefore, $P_1(Q_1-V_1)+(P_1-U_1)V_1\geq0$, and %$\overline{F}_{X_{2:m}}(x)$ is decreasing implies 
		the desired result obtains by lemma 1.
	\end{proof}	
	%As mentionsed earlier, when $\alpha=1$, the MPHR model becomes the PHR model. Therefore, the next corollary follows immediately from Theorem \ref{tp}.
	%\begin{corollary}
	%	Suppose the random variables $X_1,X_2,\dots,X_n$ follows the PHR $(\lambda_1,\lambda_2,\dots,\lambda_n;\overline{F};\psi)$ model. Further, suppose the random variables $Y_1,Y_2,\dots,Y_n$ follow the PHR $(\mu_1,\mu_2,\dots,\mu_n;\overline{F};\psi)$ model. Under the setup of Theorem \ref{tp},then
	%	\begin{equation*}
		%			\lambda\overset{\rm w}{\succeq}\mu\Rightarrow X_{2:N_1}\geq_{\rm st}Y_{2:N_2}
		%	\end{equation*}
	%\end{corollary}
	\begin{remark}
		The result of Theorem \ref{tp} improves Theorem 1 of \cite{zhang2021ordering} from the extreme order statistics to the case of the second-order statistics. 
	\end{remark}
	
	Next, we give a numerical example to illustrate the result of Theorem \ref{tp}.
	\begin{example}
		Consider the case of $n=4$. Let $\overline{F}(x)=e^{-(ax)^b}, a>0, b>0$, and generator $\psi(x)=e^{\frac{1-e^x}{\theta}}, 0<\theta\leq1.$ Set $\alpha=0.8, a=1.2, b=0.5, \theta=0.1,\lambda=(0.2,0.4,0.8,1.3)\overset{\rm w}{\succeq}(0.3,0.3,1.5,1.6)=\mu.$ Let $N_1$ be an integer-valued random variable having the probability distribution $\mathbb{P}(N_1=1)=0.05, \mathbb{P}(N_1=2)=0.2, \mathbb{P}(N_1=3)=0.3, \mathbb{P}(N_1=4)=0.45$, and $N_2$ be another integer-valued random variable with the probability distribution $\mathbb{P}(N_2=1)=0.05, \mathbb{P}(N_2=2)=0.2, \mathbb{P}(N_2=3)=0.35, \mathbb{P}(N_2=4)=0.4$. It is easy to check that the conditions of Theorem \ref{tp} are all statisfied. The plots of survival functions of $X_{2:4}$ and $Y_{2:4}$, denoted as $\overline{F}_{X_{2:4}}$ and $\overline{F}_{Y_{2:4}}$, are plotted in figure 1, all $x=-{\rm ln}(u)$ and $u\in (0,1]$. Thus, the effectiveness of Theorem \ref{tp} are vaidated.
		\begin{figure}[H]
			\centering
			%  \setlength{\abovecaptionskip}{-0.1cm}   
			%  \vspace{0.1cm}  
			%  \subfigcapskip=-8pt  
			
			{\includegraphics[width=7cm]{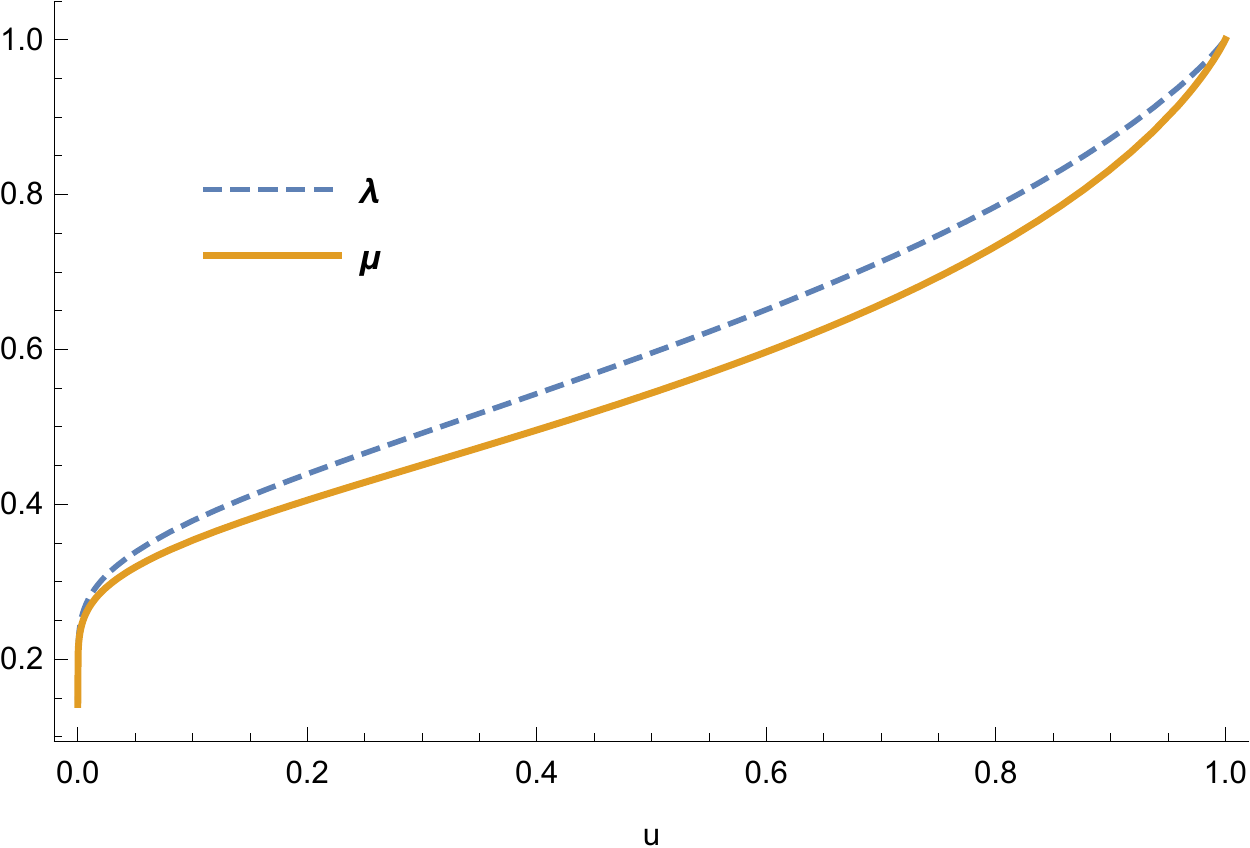}\label{fig2a}}
			
			\caption{Plots of  survival function of $X_{2:4}[Y_{2:4}]$, for all $x=-\ln u$ and  $u\in(0,1]$.
			}
		\end{figure}
	\end{example}
	
	The following theorem presents the usual stochastic order between the second-order statistics arising from two sets of random variables with Archimedean copula, here, we assume that the two samples have common modified proportional hazard rates parameters.
	\begin{theorem}\label{tla}
		Let $X_1,X_2\dots,X_n$ be n statistically dependent heterogeneous random variables with $X_i\thicksim MPHR(\alpha_i,\lambda;\overline{F};\psi)$, where $0<\alpha_i\leq1,$ for all $i=1,2,\dots,n,$ and let $N_1$ be a non-negative integer-valued random variable independent of $ X_{i}'$s. Let $Y_1,Y_2\dots,Y_n$ be another n statistically dependent heterogeneous random variables with $Y_i\thicksim MPHR(\beta_i,\lambda;\overline{F};\psi)$, where $0<\beta_i\leq1,$ for all $i=1,2,\dots,n,$ and let $N_2$ be another non-negative integer-valued random variable independent of $ Y_{i}'$s. Suppose $\alpha,\beta\in\mathcal D_+$ $(         or
		\ \mathcal{I_+})$, if $N_1\geq_{\rm st}N_2,$ and $\psi$ is log-concave. Then we have
		\begin{equation*}
			\frac{1}{\alpha}\overset{\rm w}{\succeq}\frac{1}{\beta}\Rightarrow X_{2:N_1}\geq_{\rm st}Y_{2:N_2}
		\end{equation*}
	\end{theorem}

	\begin{proof}
		We present the proof for $\alpha,\beta\in \mathcal{I}_+$. The proof for the other case is similar, and hence not presented here. The survival function of $X_{2:n}$ can be expressed as 
		\begin{equation*}
			\overline{F}_{X_{2:n}}(x)=\sum\limits_{i=1}^{n}\psi\left( \sum\limits_{j\neq i}^{n}\phi\left( \frac{\alpha_j\overline{F}^{\lambda}(x)}{1-\overline{\alpha_j}\overline{F}^{\lambda}(x)}\right) \right) -(n-1)\psi\left(\sum\limits_{i=1}^{n}\phi\left(  \frac{\alpha_i\overline{F}^{\lambda}(x)}{1-\overline{\alpha_i}\overline{F}^{\lambda}(x)}\right)  \right).
		\end{equation*}
		Similarly, to obtain the desired result, it is then sufficient to show that $\overline{F}_{X_{2:m}}(x)\geq\overline{F}_{Y_{2:m}}(x),\ m=1,\dots,n.$
		Let $a_k=\frac{1}{\alpha_k}$, for $k=1,\dots,m$. On differentiating this partially with respect to $a_k,$ we get
		\begin{eqnarray*}
			&&\frac{\partial\overline{F}_{X_{2:m}}(x)}{\partial a_k}\\
			&=&\frac{\psi\left( \phi\left( \frac{\frac{1}{a_k}\overline{F}^{\lambda}(x)}{1-(1-\frac{1}{a_k})\overline{F}^{\lambda}(x)}\right)\right) }{\psi^\prime\left( \phi\left( \frac{\frac{1}{a_k}\overline{F}^{\lambda}(x)}{1-(1-\frac{1}{a_k})\overline{F}^{\lambda}(x)}\right)\right)}\frac{-\frac{1}{a_k}[1-\overline{F}^{\lambda}(x)]}{1-(1-\frac{1}{a_k})\overline{F}^{\lambda}(x)}\\
			&&\times\left[ \sum\limits_{i\neq k}^{m}\psi^\prime\left( \sum\limits_{j\neq i}^{m}\phi\left( \frac{\frac{1}{a_j}\overline{F}^{\lambda}(x)}{1-(1-\frac{1}{a_j})\overline{F}^{\lambda}(x)}\right) \right) -(m-1)\psi^\prime\left(\sum\limits_{i=1}^{m}\phi\left( \frac{\frac{1}{a_i}\overline{F}^{\lambda}(x)}{1-(1-\frac{1}{a_i})\overline{F}^{\lambda}(x)}\right)\right)\right]\\
			&\leq&0.
		\end{eqnarray*}
		Notice that
		\begin{eqnarray*}
			&&\frac{\partial\overline{F}_{X_{2:m}}(x)}{\partial a_k}-\frac{\partial\overline{F}_{X_{2:m}}(x)}{\partial a_l}\\
			&=&\frac{\psi\left( \phi\left( \frac{\frac{1}{a_k}\overline{F}^{\lambda}(x)}{1-(1-\frac{1}{a_k})\overline{F}^{\lambda}(x)}\right)\right) }{\psi^\prime\left( \phi\left( \frac{\frac{1}{a_k}\overline{F}^{\lambda}(x)}{1-(1-\frac{1}{a_k})\overline{F}^{\lambda}(x)}\right)\right)}\frac{-\frac{1}{a_k}[1-\overline{F}^{\lambda}(x)]}{1-(1-\frac{1}{a_k})\overline{F}^{\lambda}(x)}\\
			&&\times\left[ \sum\limits_{i\neq k}^{m}\psi^\prime\left( \sum\limits_{j\neq i}^{m}\phi\left( \frac{\frac{1}{a_j}\overline{F}^{\lambda}(x)}{1-(1-\frac{1}{a_j})\overline{F}^{\lambda}(x)}\right) \right) -(m-1)\psi^\prime\left(\sum\limits_{i=1}^{m}\phi\left( \frac{\frac{1}{a_i}\overline{F}^{\lambda}(x)}{1-(1-\frac{1}{a_i})\overline{F}^{\lambda}(x)}\right)\right)\right]\\
			&&-\frac{\psi\left( \phi\left( \frac{\frac{1}{a_l}\overline{F}^{\lambda}(x)}{1-(1-\frac{1}{a_l})\overline{F}^{\lambda}(x)}\right)\right) }{\psi^\prime\left( \phi\left( \frac{\frac{1}{a_l}\overline{F}^{\lambda}(x)}{1-(1-\frac{1}{a_l})\overline{F}^{\lambda}(x)}\right)\right)}\frac{-\frac{1}{a_l}[1-\overline{F}^{\lambda}(x)]}{1-(1-\frac{1}{a_l})\overline{F}^{\lambda}(x)}\\
			&&\times\left[ \sum\limits_{i\neq l}^{m}\psi^\prime\left( \sum\limits_{j\neq i}^{m}\phi\left( \frac{\frac{1}{a_j}\overline{F}^{\lambda}(x)}{1-(1-\frac{1}{a_j})\overline{F}^{\lambda}(x)}\right) \right) -(m-1)\psi^\prime\left(\sum\limits_{i=1}^{m}\phi\left( \frac{\frac{1}{a_i}\overline{F}^{\lambda}(x)}{1-(1-\frac{1}{a_i})\overline{F}^{\lambda}(x)}\right)\right)\right]\\
			&=:&P_2Q_2-U_2V_2=P_2(Q_2-V_2)+(P_2-U_2)V_2.
		\end{eqnarray*}
		For any $1\leq k<l\leq m$, $a_k\geq a_l$, by the log-concave and decreasing properties of $\psi$, we have
		\begin{eqnarray*}
			\frac{\psi\left( \phi\left( \frac{\frac{1}{a_k}\overline{F}^{\lambda}(x)}{1-(1-\frac{1}{a_k})\overline{F}^{\lambda}(x)}\right)\right) }{\psi^\prime\left( \phi\left( \frac{\frac{1}{a_k}\overline{F}^{\lambda}(x)}{1-(1-\frac{1}{a_k})\overline{F}^{\lambda}(x)}\right)\right)}\geq \frac{\psi\left( \phi\left( \frac{\frac{1}{a_l}\overline{F}^{\lambda}(x)}{1-(1-\frac{1}{a_l})\overline{F}^{\lambda}(x)}\right)\right) }{\psi^\prime\left( \phi\left( \frac{\frac{1}{a_l}\overline{F}^{\lambda}(x)}{1-(1-\frac{1}{a_l})\overline{F}^{\lambda}(x)}\right)\right)}.
		\end{eqnarray*}
		It holds that 
		\begin{eqnarray*}
			&&(P_2-U_2)V_2\\
			&\overset{\rm sgn}{=}&\frac{\psi\left( \phi\left( \frac{\frac{1}{a_k}\overline{F}^{\lambda}(x)}{1-(1-\frac{1}{a_k})\overline{F}^{\lambda}(x)}\right)\right) }{\psi^\prime\left( \phi\left( \frac{\frac{1}{a_k}\overline{F}^{\lambda}(x)}{1-(1-\frac{1}{a_k})\overline{F}^{\lambda}(x)}\right)\right)}\frac{\frac{1}{a_k}}{1-(1-\frac{1}{a_k})\overline{F}^{\lambda}(x)}-\frac{\psi\left( \phi\left( \frac{\frac{1}{a_l}\overline{F}^{\lambda}(x)}{1-(1-\frac{1}{a_l})\overline{F}^{\lambda}(x)}\right)\right) }{\psi^\prime\left( \phi\left( \frac{\frac{1}{a_l}\overline{F}^{\lambda}(x)}{1-(1-\frac{1}{a_l})\overline{F}^{\lambda}(x)}\right)\right)}\frac{\frac{1}{a_l}}{1-(1-\frac{1}{a_l})\overline{F}^{\lambda}(x)}\\
			&=&\frac{\psi\left( \phi\left( \frac{\frac{1}{a_k}\overline{F}^{\lambda}(x)}{1-(1-\frac{1}{a_k})\overline{F}^{\lambda}(x)}\right)\right) }{\psi^\prime\left( \phi\left( \frac{\frac{1}{a_k}\overline{F}^{\lambda}(x)}{1-(1-\frac{1}{a_k})\overline{F}^{\lambda}(x)}\right)\right)}\left[ \frac{\frac{1}{a_k}}{1-(1-\frac{1}{a_k})\overline{F}^{\lambda}(x)}-\frac{\frac{1}{a_l}}{1-(1-\frac{1}{a_l})\overline{F}^{\lambda}(x)}\right]\\
			&&+\frac{\frac{1}{a_l}}{1-(1-\frac{1}{a_l})\overline{F}^{\lambda}(x)}\left[ \frac{\psi\left( \phi\left( \frac{\frac{1}{a_k}\overline{F}^{\lambda}(x)}{1-(1-\frac{1}{a_k})\overline{F}^{\lambda}(x)}\right)\right) }{\psi^\prime\left( \phi\left( \frac{\frac{1}{a_k}\overline{F}^{\lambda}(x)}{1-(1-\frac{1}{a_k})\overline{F}^{\lambda}(x)}\right)\right)}-\frac{\psi\left( \phi\left( \frac{\frac{1}{a_l}\overline{F}^{\lambda}(x)}{1-(1-\frac{1}{a_l})\overline{F}^{\lambda}(x)}\right)\right) }{\psi^\prime\left( \phi\left( \frac{\frac{1}{a_l}\overline{F}^{\lambda}(x)}{1-(1-\frac{1}{a_l})\overline{F}^{\lambda}(x)}\right)\right)}\right]\\  
			&\geq&0.
		\end{eqnarray*}
		For any $a_k\geq a_l$, owing to $\phi$ is convex and decreasing, we have 
		\begin{eqnarray*}
			\psi^\prime\left( \sum\limits_{j\neq k}^{m}\phi\left( \frac{\frac{1}{a_j}\overline{F}^{\lambda}(x)}{1-(1-\frac{1}{a_j})\overline{F}^{\lambda}(x)}\right) \right)\leq\psi^\prime\left( \sum\limits_{j\neq l}^{m}\phi\left( \frac{\frac{1}{a_j}\overline{F}^{\lambda}(x)}{1-(1-\frac{1}{a_j})\overline{F}^{\lambda}(x)}\right) \right).
		\end{eqnarray*}
		As a result, we have
		\begin{eqnarray*}
			P_2(Q_2-V_2)
			&\overset{\rm sgn}{=}& \sum\limits_{i\neq k}^{m}\psi^\prime\left( \sum\limits_{j\neq i}^{m}\phi\left( \frac{\frac{1}{a_j}\overline{F}^{\lambda}(x)}{1-(1-\frac{1}{a_j})\overline{F}^{\lambda}(x)}\right) \right) - \sum\limits_{i\neq l}^{m}\psi^\prime\left( \sum\limits_{j\neq i}^{m}\phi\left( \frac{\frac{1}{a_j}\overline{F}^{\lambda}(x)}{1-(1-\frac{1}{a_j})\overline{F}^{\lambda}(x)}\right) \right)\\
			&=&\psi^\prime\left( \sum\limits_{j\neq l}^{m}\phi\left( \frac{\frac{1}{a_j}\overline{F}^{\lambda}(x)}{1-(1-\frac{1}{a_j})\overline{F}^{\lambda}(x)}\right) \right)-\psi^\prime\left( \sum\limits_{j\neq k}^{m}\phi\left( \frac{\frac{1}{a_j}\overline{F}^{\lambda}(x)}{1-(1-\frac{1}{a_j})\overline{F}^{\lambda}(x)}\right) \right)\\
			&\geq&0.
		\end{eqnarray*}
		Therefore, the desires result completed by lemma 1.

	\end{proof}
	\begin{remark}
		The result of Theorem \ref{tla} improves Theorem 2 of \cite{zhang2021ordering} from the extreme order statistics to the case of the second-order statistics. 
	\end{remark}
	The next example illustrates result of Theorem \ref{tla}.
	\begin{example}
		Consider the case of $n=4$. Let $\overline{F}(x)=e^{-(ax)^b}, a>0, b>0$, and generator $\psi(x)=e^{1-(1+x)^\theta}, \theta>0.$ Taking $\lambda=0.4, a=0,5, b=0.8, \theta=7,\alpha=(\frac{1}{3},\frac{1}{3},\frac{1}{5},\frac{1}{8})\ and\ \beta=(\frac{1}{5},\frac{1}{6},\frac{1}{7},\frac{1}{9}).$ Evidently, $(\frac{1}{\alpha_1},\frac{1}{\alpha_2},\frac{1}{\alpha_3},\frac{1}{\alpha_4})\overset{\rm w}{\succeq}(\frac{1}{\beta_1},\frac{1}{\beta_2},\frac{1}{\beta_3},\frac{1}{\beta_4}).$ Let $N_1$ be an integer-valued random variable having the probability distribution $\mathbb{P}(N_1=1)=0.05, \mathbb{P}(N_1=2)=0.2, \mathbb{P}(N_1=3)=0.3, \mathbb{P}(N_1=4)=0.45$, and $N_2$ be another integer-valued random variable with the probability distribution $\mathbb{P}(N_2=1)=0.05, \mathbb{P}(N_2=2)=0.2, \mathbb{P}(N_2=3)=0.35, \mathbb{P}(N_2=4)=0.4$, it is apparently  that $N_1\geq_{\rm st}N_2$. The plots of survival functions of $X_{2:4}$ and $Y_{2:4}$, denoted as $\overline{F}_{X_{2:4}}$ and $\overline{F}_{Y_{2:4}}$, are plotted in figure 2, all $x=-{\rm ln}(u)$ and $u\in (0,1]$. Thus, the effectiveness of Theorem \ref{tp} are vaidated.
		\begin{figure}[H]
			\centering
			%  \setlength{\abovecaptionskip}{-0.1cm}   
			%  \vspace{0.1cm}  
			%  \subfigcapskip=-8pt  
			
			{\includegraphics[width=7cm]{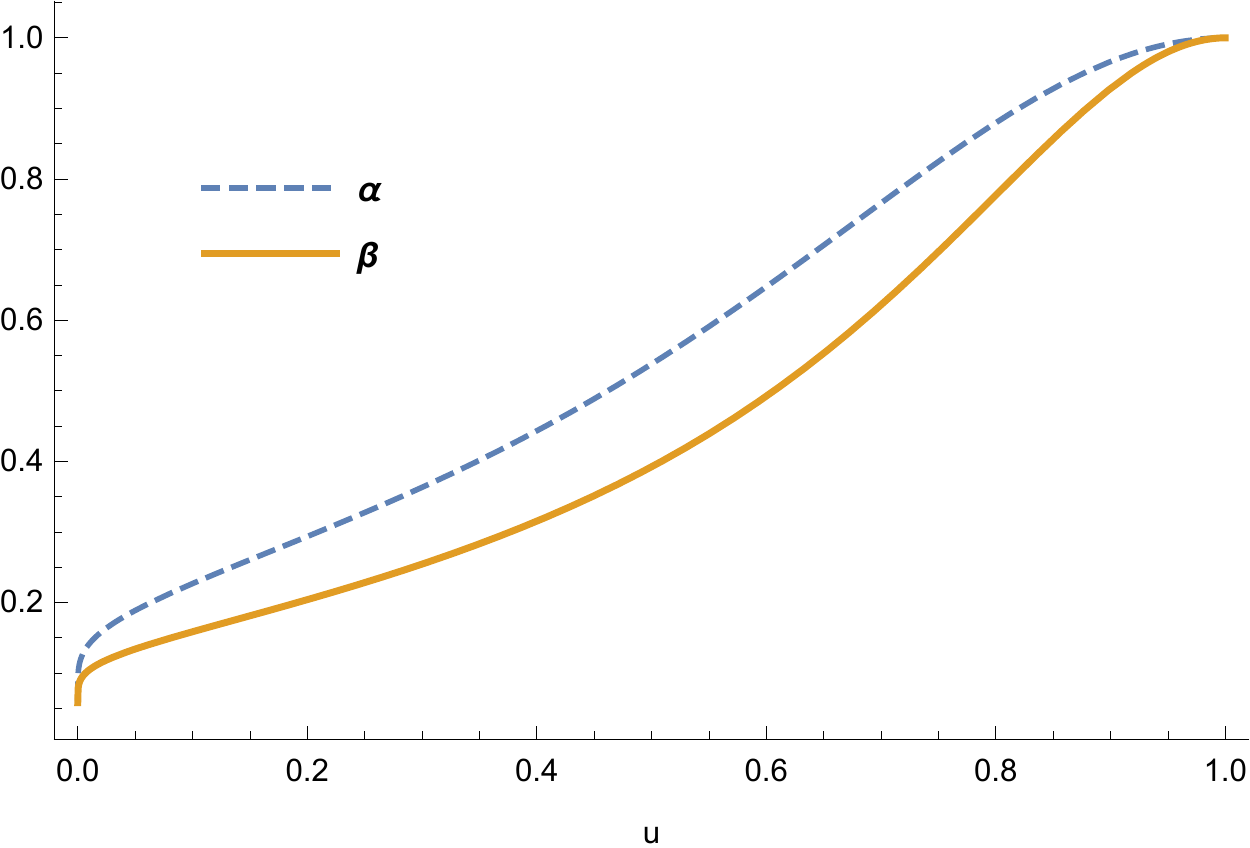}\label{fig2a}}
			
			\caption{Plots of  survival function of $X_{2:4}[Y_{2:4}]$, for all $x=-\ln u$ and  $u\in(0,1]$.
			}
		\end{figure}
	\end{example}

	\section{Hazard rate order of second-order statistics from independent and heterogeneous  observations}\label{hr}
	The next theorem gives sufficient conditions guaranteeing the hazard rate order between two
	modified proportional hazard rates models with the same modified proportional hazard rates parameters
	and the heterogeneous tilt parameters.
	
	\begin{theorem}\label{tpla}
		Let $X_1,X_2\dots,X_n$ be n independent heterogeneous random variables with $X_i\thicksim MPHR(\alpha_i,\lambda;\overline{F})$, where $0<\alpha_i\leq1,$ for all $i=1,2,\dots,n.$ Let $Y_1,Y_2\dots,Y_n$ be another n independent heterogeneous random variables with $Y_i\thicksim MPHR(\beta_i,\lambda;\overline{F})$, where $0<\beta_i\leq1,$ for all $i=1,2,\dots,n$. Suppose $\alpha,\beta\in\mathcal D_+$ $(         or
		\ \mathcal{I_+})$. Then we have
		
		%Let  $X_i\thicksim MPHR(\alpha_i,\lambda;\overline{F})$ and $Y_i\thicksim MPHR(\beta_i,\lambda;\overline{F})$, where $0\leq\alpha_i\leq1,0\leq\beta_i\leq1,$for all $i=1,2,\dots n$. Let $\alpha,\beta\in\mathcal D_+$ $(         or \ \mathcal{I_+})$  Then we have
		\begin{equation*}
			\frac{1}{\alpha}\overset{\rm m}{\succeq}\frac{1}{\beta}\Rightarrow X_{2:n}\geq_{\rm hr}Y_{2:n}.
		\end{equation*}

	\end{theorem}
	
	\begin{proof}
		The idea of the proof is borrowed from Theorem 3.7 of \cite{das2021comparison}.	
		The survival function of $X_{2:n}$ can be rewritten 
		\begin{eqnarray*}
			\overline{F}_{X_{2:n}}(x)&=&\sum\limits^n_{i=1}\prod\limits^n_{j\neq i}\frac{\alpha_j\overline{F}^{\lambda}(x)}{1-\overline{\alpha_j}\overline{F}^{\lambda}(x)}-(n-1)\prod\limits^n_{i=1}\frac{\alpha_i\overline{F}^{\lambda}(x)}{1-\overline{\alpha_i}\overline{F}^{\lambda}(x)}\\
			&=&\prod\limits^n_{i=1}\frac{\alpha_i\overline{F}^{\lambda}(x)}{1-\overline{\alpha_i}\overline{F}^{\lambda}(x)}\left[ \sum\limits^n_{i=1}\frac{1-\overline{\alpha_i}\overline{F}^{\lambda}(x)}{\alpha_i\overline{F}^{\lambda}(x)}-(n-1)\right]\\
			&=&\prod\limits^n_{i=1}\frac{\alpha_i\overline{F}^{\lambda}(x)}{1-\overline{\alpha_i}\overline{F}^{\lambda}(x)}\left[ \sum\limits^n_{i=1}\frac{1-\overline{F}^{\lambda}(x)}{\alpha_i\overline{F}^{\lambda}(x)}+1\right]. 
		\end{eqnarray*}
		Thus, the hazard rate function of the second-order statistic $X_{2:n}$ can be obtained as 
		\begin{eqnarray*}
			r_{X_{2:n}}(x)=\sum\limits^n_{i=1}\frac{\lambda r(x)}{1-\overline{\alpha}_i\overline{F}^{\lambda}(x)}-\frac{\sum\limits^n_{i=1}\frac{\lambda r(x)}{\alpha_i\overline{F}^{\lambda}(x)}}{\sum\limits^n_{i=1}\frac{1-\overline{F}^\lambda(x)}{\alpha_i\overline{F}^{\lambda}(x)}+1},
		\end{eqnarray*}
		where $r(x)$ is the hazard rate function of the baseline distribution. Define
		\begin{eqnarray}\label{one}
			\Phi(a)=r_{X_{2:n}}(x)=\sum\limits^n_{i=1}\frac{\lambda r(x)}{1-(1-\frac{1}{a_i})\overline{F}^{\lambda}(x)}-\frac{\sum\limits^n_{i=1}\frac{a_i\lambda r(x)}{\overline{F}^{\lambda}(x)}}{\sum\limits^n_{i=1}\frac{a_i(1-\overline{F}^\lambda(x))}{\overline{F}^{\lambda}(x)}+1},
		\end{eqnarray}
		where $a_i=\frac{1}{\alpha_i}$, $i=1,\dots,n.$ After differentiating equation (\ref{one}) partially with respect to $a_k,$ we obtain
		\begin{eqnarray}\label{two}
			\frac{\partial \Phi(a) }{\partial a_k}=\frac{\frac{1}{a^2_k}\overline{F}^{\lambda}(x)\lambda r(x)}{\left[1-(1-\frac{1}{a_k})\overline{F}^{\lambda}(x) \right]^2}-\frac{\frac{\lambda r(x)}{\overline{F}^{\lambda}(x)}}{\sum\limits^n_{i=1}\frac{a_i(1-\overline{F}^\lambda(x))}{\overline{F}^{\lambda}(x)}+1}+\frac{\sum\limits^n_{i=1}\frac{a_i\lambda r(x)}{\overline{F}^{\lambda}(x)}\frac{1-\overline{F}^{\lambda}(x)}{\overline{F}^{\lambda}(x)}}{\left[ \sum\limits^n_{i=1}\frac{a_i(1-\overline{F}^\lambda(x))}{\overline{F}^{\lambda}(x)}+1\right]^2 }.
		\end{eqnarray}
		To prove the result, it is sufficient to show that $\Phi(a)$ is schur-concave with respect to $a\in\mathcal{I_+}$ $or\mathcal{D_+}$. Without loss of generality, let $a\in\mathcal{I_+}$, for any $1\leq k<l\leq n,$ $a_k\leq a_l$. Using (\ref{two}), we have
		\begin{eqnarray*}
			\frac{\partial \Phi(a) }{\partial a_k}-\frac{\partial \Phi(a) }{\partial a_l}&\overset{\rm sgn}{=}&\frac{\frac{1}{a^2_k}}{\left[1-(1-\frac{1}{a_k})\overline{F}^{\lambda}(x) \right]^2}-\frac{\frac{1}{a^2_l}}{\left[1-(1-\frac{1}{a_l})\overline{F}^{\lambda}(x) \right]^2}\geq0,
		\end{eqnarray*}
		where the last inequality is due to lemma \ref{proof}. Thus, the result follows from Lemma \ref{kundu}.

	\end{proof}
	The following example illustrates the result of Theorem \ref{tpla}.
	\begin{remark}
		The result of Theorem \ref{tpla} completes Theorem 3 of \cite{zhang2021ordering} from the extreme order statistics to the case of the second-order statistics. 
	\end{remark}
	\begin{example}
		Let $\overline{F}(x)=e^{-(ax)^b}, a>0, b>0$, Set $n=4, \lambda=0.5, a=0.15, b=1.2, \alpha=(\frac{1}{4},\frac{1}{3},\frac{1}{2},1)\ and\ \beta=(\frac{1}{3},\frac{1}{3},\frac{1}{2},\frac{1}{2}).$ Evidently, $(\frac{1}{\alpha_1},\frac{1}{\alpha_2},\frac{1}{\alpha_3},\frac{1}{\alpha_4})\overset{\rm m}{\succeq}(\frac{1}{\beta_1},\frac{1}{\beta_2},\frac{1}{\beta_3},\frac{1}{\beta_4}).$ One can verify that the conditions of Theorem \ref{tpla} are all statisfied. Figure 3 plots the hazard rate functions of $X_{2:4}$ and $Y_{2:4}$, denoted as $r_{X_{2:4}}(x)$ and $r_{Y_{2:4}}(x)$, all $x=-{\rm ln}(u)$ and $u\in (0,1]$. Thus, we can get $X_{2:4}\geq_{\rm hr}Y_{2:4}$.
		\begin{figure}[H]
			\centering
			%  \setlength{\abovecaptionskip}{-0.1cm}   
			%  \vspace{0.1cm}  
			%  \subfigcapskip=-8pt  
			
			{\includegraphics[width=7cm]{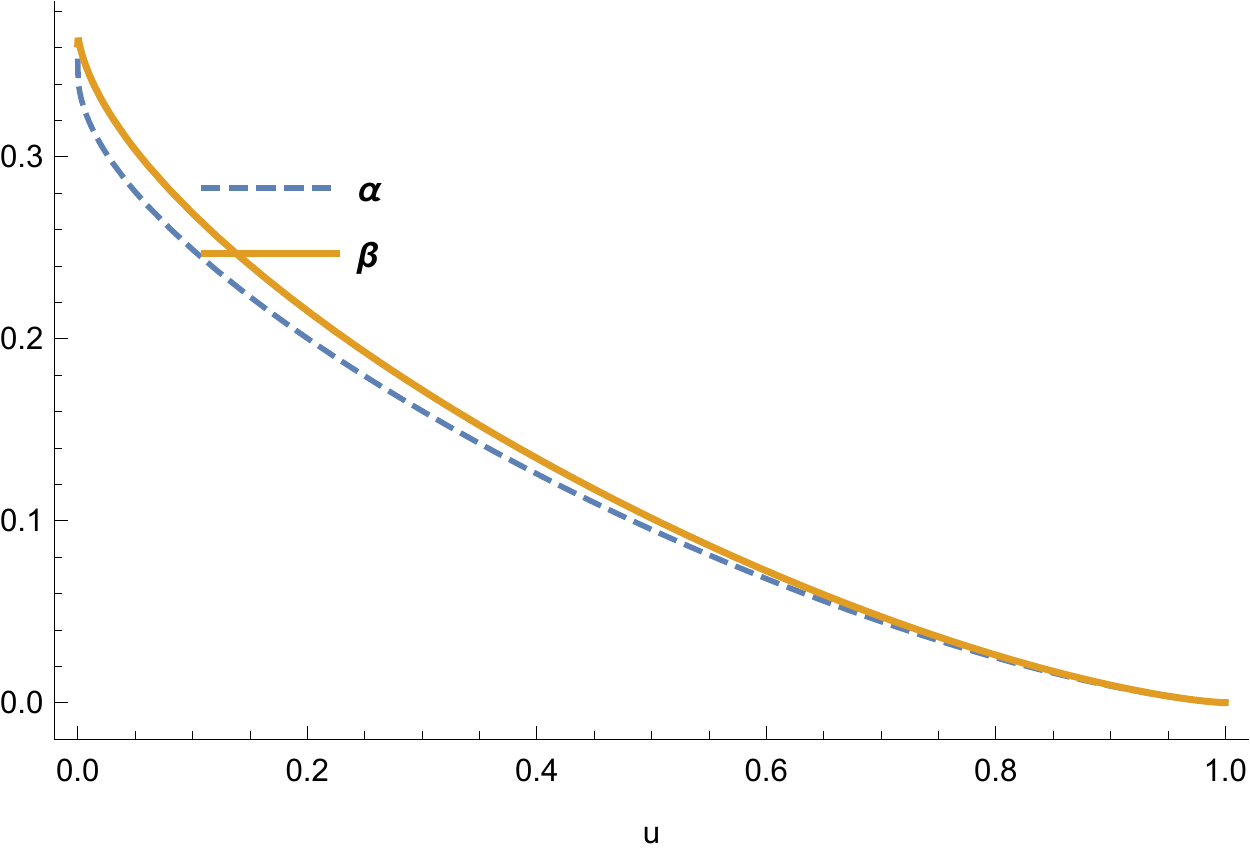}\label{fig2a}}
			
			\caption{Plots of  hazard rate functions of $X_{2:4}[Y_{2:4}]$, for all $x=-\ln u$ and  $u\in(0,1]$.
			}
		\end{figure}
	\end{example}
	
	In the next, we discuss stochastic comparisons on the second-order statistics arising from multiple-outlier MPHR samples in the sense of the hazard rate order. We first present the following result.
	\begin{theorem}\label{thrp}
		Let $X_{1}, X_{2}, \dots,X_{n}$ be independent random variables following the multiple-outlier MPHR moder with survival functions$(\frac{\alpha\overline{F}^{\lambda_1}(x)}{1-\overline{\alpha}\overline{F}^{\lambda_1}(x)}{\rm I}_p,\frac{\alpha\overline{F}^{\lambda}(x)}{1-\overline{\alpha}\overline{F}^{\lambda}(x)}{\rm I}_q),$ where $p+q=n,p,q\geq1,$ $\overline{F}(x)$ is the baseline survival function, and ${\rm I}_p$ stand for a p-dimensional vector with all its components equal to 1.Let $Y_{1}, Y_{2}, \dots,Y_{n}$ be another set of independent random variables following the multiple-outlier MPHR moder with survival functions$(\frac{\alpha\overline{F}^{\lambda_2}(x)}{1-\overline{\alpha}\overline{F}^{\lambda_2}(x)}{\rm I}_p,\frac{\alpha\overline{F}^{\lambda}(x)}{1-\overline{\alpha}\overline{F}^{\lambda}(x)}{\rm I}_q).$ If $\lambda\geq\lambda_2\geq\lambda_1,$ then $X_{2:n}\geq_{\rm hr}Y_{2:n}.$

	\end{theorem}
	
	\begin{proof}
		The idea of the proof is borrowed from Theorem 3.1 of \cite{cai2017hazard}.
		The survival function of $X_{2:n}$ and $Y_{2:n}$ are given by 
		\begin{eqnarray*}
			\overline{F}_{X_{2:n}}(x)&=&p\left[ \frac{\alpha\overline{F}^{\lambda_1}(x)}{1-\overline{\alpha}\overline{F}^{\lambda_1}(x)}\right]^{p-1}\left[ \frac{\alpha\overline{F}^{\lambda}(x)}{1-\overline{\alpha}\overline{F}^{\lambda}(x)}\right]^q+q\left[ \frac{\alpha\overline{F}^{\lambda_1}(x)}{1-\overline{\alpha}\overline{F}^{\lambda_1}(x)}\right]^p\left[ \frac{\alpha\overline{F}^{\lambda}(x)}{1-\overline{\alpha}\overline{F}^{\lambda}(x)}\right]^{q-1}\\
			&&-(n-1)\left[ \frac{\alpha\overline{F}^{\lambda_1}(x)}{1-\overline{\alpha}\overline{F}^{\lambda_1}(x)}\right]^p\left[ \frac{\alpha\overline{F}^{\lambda}(x)}{1-\overline{\alpha}\overline{F}^{\lambda}(x)}\right]^q,      
		\end{eqnarray*}
		and
		\begin{eqnarray*}
			\overline{F}_{X^*_{2:n}}(x)&=&p\left[ \frac{\alpha\overline{F}^{\lambda_2}(x)}{1-\overline{\alpha}\overline{F}^{\lambda_2}(x)}\right]^{p-1}\left[ \frac{\alpha\overline{F}^{\lambda}(x)}{1-\overline{\alpha}\overline{F}^{\lambda}(x)}\right]^q+q\left[ \frac{\alpha\overline{F}^{\lambda_2}(x)}{1-\overline{\alpha}\overline{F}^{\lambda_2}(x)}\right]^p\left[ \frac{\alpha\overline{F}^{\lambda}(x)}{1-\overline{\alpha}\overline{F}^{\lambda}(x)}\right]^{q-1}\\
			&&-(n-1)\left[ \frac{\alpha\overline{F}^{\lambda_2}(x)}{1-\overline{\alpha}\overline{F}^{\lambda_2}(x)}\right]^p\left[ \frac{\alpha\overline{F}^{\lambda}(x)}{1-\overline{\alpha}\overline{F}^{\lambda}(x)}\right]^q,  
		\end{eqnarray*}
		where $\overline{\alpha}=1-\alpha$ and $p+q=n.$ Let us set $\overline{F}^{\lambda}(x)=e^{-\lambda(-{\rm ln}\overline{F}(x))}$ and $t=-{\rm ln}\overline{F}(x)$. We can then rewritten the survival function $X_{2:n}$ as:
		\begin{eqnarray*}
			\overline{F}_{X_{2:n}}(t)&=&p\left[ \frac{\alpha}{e^{\lambda_1t}-\overline{\alpha}}\right]^{p-1}\left[ \frac{\alpha}{e^{\lambda t}-\overline{\alpha}}\right]^q+q\left[ \frac{\alpha}{e^{\lambda_1t}-\overline{\alpha}}\right]^p\left[ \frac{\alpha}{e^{\lambda t}-\overline{\alpha}}\right]^{q-1}\\
			&&-(n-1)\left[ \frac{\alpha}{e^{\lambda_1t}-\overline{\alpha}}\right]^p\left[ \frac{\alpha}{e^{\lambda t}-\overline{\alpha}}\right]^q, t\geq0.      
		\end{eqnarray*}
		Thus, its hazard rate function is given by 
		\begin{eqnarray*}
			r_{X_{2:n}}(t)&=&\frac{{\rm d}[-{\rm ln}\overline{F}_{X_{2:n}}(t)]}{{\rm d}t} =\frac{pA_1\left( \frac{e^{\lambda_1t}-\overline{\alpha}}{\alpha}\right) +qA_2\left( \frac{e^{\lambda t}-\overline{\alpha}}{\alpha}\right) -(n-1)A}{p\left( \frac{e^{\lambda_1t}-\overline{\alpha}}{\alpha}\right) +q\left( \frac{e^{\lambda t}-\overline{\alpha}}{\alpha}\right) -(n-1)},
		\end{eqnarray*}
		where $A_1=(p-1)\frac{\lambda_1e^{\lambda_1t}}{e^{\lambda_1t}-\overline{\alpha}}+q\frac{\lambda e^{\lambda t}}{e^{\lambda t}-\overline{\alpha}}$, $A_2=p\frac{\lambda_1e^{\lambda_1t}}{e^{\lambda_1t}-\overline{\alpha}}+(q-1)\frac{\lambda e^{\lambda t}}{e^{\lambda t}-\overline{\alpha}}$ and $A=p\frac{\lambda_1e^{\lambda_1t}}{e^{\lambda_1t}-\overline{\alpha}}+q\frac{\lambda e^{\lambda t}}{e^{\lambda t}-\overline{\alpha}}$.
		It suffices to show $r_{X_{2:n}}(t)\leq r_{X^*_{2:n}}(t)$ for $\lambda\geq \lambda_2\geq \lambda_1>0,$ i.e.,
		\begin{eqnarray*}
			\frac{pA_1\left( \frac{e^{\lambda_1t}-\overline{\alpha}}{\alpha}\right) +qA_2\left( \frac{e^{\lambda t}-\overline{\alpha}}{\alpha}\right) -(n-1)A}{p\left( \frac{e^{\lambda_1t}-\overline{\alpha}}{\alpha}\right) +q\left( \frac{e^{\lambda t}-\overline{\alpha}}{\alpha}\right) -(n-1)}\leq\frac{pB_1\left( \frac{e^{\lambda_2t}-\overline{\alpha}}{\alpha}\right) +qB_2\left( \frac{e^{\lambda t}-\overline{\alpha}}{\alpha}\right) -(n-1)B}{p\left( \frac{e^{\lambda_2t}-\overline{\alpha}}{\alpha}\right) +q\left( \frac{e^{\lambda t}-\overline{\alpha}}{\alpha}\right) -(n-1)}.
		\end{eqnarray*}
		where $B_1=(p-1)\frac{\lambda_2e^{\lambda_2t}}{e^{\lambda_2t}-\overline{\alpha}}+q\frac{\lambda e^{\lambda t}}{e^{\lambda t}-\overline{\alpha}}$, $B_2=p\frac{\lambda_2e^{\lambda_2t}}{e^{\lambda_2t}-\overline{\alpha}}+(q-1)\frac{\lambda e^{\lambda t}}{e^{\lambda t}-\overline{\alpha}}$ and $B=p\frac{\lambda_2e^{\lambda_2t}}{e^{\lambda_2t}-\overline{\alpha}}+q\frac{\lambda e^{\lambda t}}{e^{\lambda t}-\overline{\alpha}}$, or equivalently
		\begin{eqnarray*}
			&&\Bigg \{ p\left[ (p-1)\frac{\lambda_1e^{\lambda_1t}}{e^{\lambda_1t}-\overline{\alpha}}+q\frac{\lambda e^{\lambda t}}{e^{\lambda t}-\overline{\alpha}}\right] \left( \frac{e^{\lambda_1t}-\overline{\alpha}}{\alpha}\right)
			+q\left[ p\frac{\lambda_1e^{\lambda_1t}}{e^{\lambda_1t}-\overline{\alpha}}+(q-1)\frac{\lambda e^{\lambda t}}{e^{\lambda t}-\overline{\alpha}}\right] \left( \frac{e^{\lambda t}-\overline{\alpha}}{\alpha}\right)\\
			&&
			- (n-1)\left[ p\frac{\lambda_1e^{\lambda_1t}}{e^{\lambda_1t}-\overline{\alpha}}+q\frac{\lambda e^{\lambda t}}{e^{\lambda t}-\overline{\alpha}}\right] \Bigg \}\times  \frac{1}{p\left( \frac{e^{\lambda_1t}-\overline{\alpha}}{\alpha}\right) +q\left( \frac{e^{\lambda t}-\overline{\alpha}}{\alpha}\right) -(n-1)}\\ 
			&\leq& \Bigg \{ p\left[ (p-1)\frac{\lambda_2e^{\lambda_2t}}{e^{\lambda_2t}-\overline{\alpha}}+q\frac{\lambda e^{\lambda t}}{e^{\lambda t}-\overline{\alpha}}\right] \left( \frac{e^{\lambda_2t}-\overline{\alpha}}{\alpha}\right) 
			+q\left[ p\frac{\lambda_2e^{\lambda_2t}}{e^{\lambda_2t}-\overline{\alpha}}+(q-1)\frac{\lambda e^{\lambda t}}{e^{\lambda t}-\overline{\alpha}}\right] \left( \frac{e^{\lambda t}-\overline{\alpha}}{\alpha}\right) \\
			&&- (n-1)\left[ p\frac{\lambda_2e^{\lambda_2t}}{e^{\lambda_2t}-\overline{\alpha}}+q\frac{\lambda e^{\lambda t}}{e^{\lambda t}-\overline{\alpha}}\right] \Bigg \}\times 	\frac{1}{p\left( \frac{e^{\lambda_2t}-\overline{\alpha}}{\alpha}\right) +q\left( \frac{e^{\lambda t}-\overline{\alpha}}{\alpha}\right) -(n-1)}.	  
		\end{eqnarray*}
		Denote by
		\begin{eqnarray*}
			M_1&=&\frac{1}{p\left( \frac{e^{\lambda_1t}-\overline{\alpha}}{\alpha}\right) +q\left( \frac{e^{\lambda t}-\overline{\alpha}}{\alpha}\right) -(n-1)}\times \Bigg \{ p\left[ (p-1)\frac{\lambda_1e^{\lambda_1t}}{e^{\lambda_1t}-\overline{\alpha}}+q\frac{\lambda e^{\lambda t}}{e^{\lambda t}-\overline{\alpha}}\right] \left( \frac{e^{\lambda_1t}-\overline{\alpha}}{\alpha}\right)\\
			&&+q\left[ p\frac{\lambda_1e^{\lambda_1t}}{e^{\lambda_1t}-\overline{\alpha}}+(q-1)\frac{\lambda e^{\lambda t}}{e^{\lambda t}-\overline{\alpha}}\right] \left( \frac{e^{\lambda t}-\overline{\alpha}}{\alpha}\right)- (n-1)\left[ p\frac{\lambda_1e^{\lambda_1t}}{e^{\lambda_1t}-\overline{\alpha}}+q\frac{\lambda e^{\lambda t}}{e^{\lambda t}-\overline{\alpha}}\right]\Bigg \},
		\end{eqnarray*}
		\begin{eqnarray*}
			M_2&=&\frac{1}{p\left( \frac{e^{\lambda_2t}-\overline{\alpha}}{\alpha}\right) +q\left( \frac{e^{\lambda t}-\overline{\alpha}}{\alpha}\right) -(n-1)}\times \Bigg \{ p\left[ (p-1)\frac{\lambda_1e^{\lambda_1t}}{e^{\lambda_1t}-\overline{\alpha}}+q\frac{\lambda e^{\lambda t}}{e^{\lambda t}-\overline{\alpha}}\right] \left( \frac{e^{\lambda_2t}-\overline{\alpha}}{\alpha}\right)\\
			&&+q\left[ p\frac{\lambda_1e^{\lambda_1t}}{e^{\lambda_1t}-\overline{\alpha}}+(q-1)\frac{\lambda e^{\lambda t}}{e^{\lambda t}-\overline{\alpha}}\right] \left( \frac{e^{\lambda t}-\overline{\alpha}}{\alpha}\right)- (n-1)\left[ p\frac{\lambda_1e^{\lambda_1t}}{e^{\lambda_1t}-\overline{\alpha}}+q\frac{\lambda e^{\lambda t}}{e^{\lambda t}-\overline{\alpha}}\right]\Bigg \},
		\end{eqnarray*}
		\begin{eqnarray*}
			M_3&=&\frac{1}{p\left( \frac{e^{\lambda_2t}-\overline{\alpha}}{\alpha}\right) +q\left( \frac{e^{\lambda t}-\overline{\alpha}}{\alpha}\right) -(n-1)}\times \Bigg \{p\left[ (p-1)\frac{\lambda_2e^{\lambda_2t}}{e^{\lambda_2t}-\overline{\alpha}}+q\frac{\lambda e^{\lambda t}}{e^{\lambda t}-\overline{\alpha}}\right] \left( \frac{e^{\lambda_2t}-\overline{\alpha}}{\alpha}\right)\\
			&&+q\left[ p\frac{\lambda_2e^{\lambda_2t}}{e^{\lambda_2t}-\overline{\alpha}}+(q-1)\frac{\lambda e^{\lambda t}}{e^{\lambda t}-\overline{\alpha}}\right] \left( \frac{e^{\lambda t}-\overline{\alpha}}{\alpha}\right)
			-(n-1)\left[ p\frac{\lambda_2e^{\lambda_2t}}{e^{\lambda_2t}-\overline{\alpha}}+q\frac{\lambda e^{\lambda t}}{e^{\lambda t}-\overline{\alpha}}\right]\Bigg \}.
		\end{eqnarray*}
		For $\lambda\geq\lambda_2\geq\lambda_1>0$, by the increasing properties of $\frac{\lambda e^{\lambda t}}{e^{\lambda t}-\overline{\alpha}}$ with respect to $\lambda$, we have 
		\begin{eqnarray*}
			\frac{\lambda_2 e^{\lambda_2 t}}{e^{\lambda_2 t}-\overline{\alpha}}\geq\frac{\lambda_1 e^{\lambda_1 t}}{e^{\lambda_1 t}-\overline{\alpha}}.
		\end{eqnarray*}
		It holds that 
		\begin{eqnarray*}
			M_3-M_2&\overset{\rm sgn}{=}&p(p-1)\left( \frac{e^{\lambda_2t}-\overline{\alpha}}{\alpha}\right)\left[ \frac{\lambda_2 e^{\lambda_2 t}}{e^{\lambda_2 t}-\overline{\alpha}}-\frac{\lambda_1 e^{\lambda_1 t}}{e^{\lambda_1 t}-\overline{\alpha}}\right]+pq\left( \frac{e^{\lambda t}-\overline{\alpha}}{\alpha}\right)\times \left[ \frac{\lambda_2 e^{\lambda_2 t}}{e^{\lambda_2 t}-\overline{\alpha}}-\frac{\lambda_1 e^{\lambda_1 t}}{e^{\lambda_1 t}-\overline{\alpha}}\right]\\
			&&-p(n-1)\left[ \frac{\lambda_2 e^{\lambda_2 t}}{e^{\lambda_2 t}-\overline{\alpha}}-\frac{\lambda_1 e^{\lambda_1 t}}{e^{\lambda_1 t}-\overline{\alpha}}\right]\\
			&=&p\left[ \frac{\lambda_2 e^{\lambda_2 t}}{e^{\lambda_2 t}-\overline{\alpha}}-\frac{\lambda_1 e^{\lambda_1 t}}{e^{\lambda_1 t}-\overline{\alpha}}\right]\left[ (p-1)\left( \frac{e^{\lambda_2t}-\overline{\alpha}}{\alpha}\right)+q\left( \frac{e^{\lambda t}-\overline{\alpha}}{\alpha}\right)-(n-1)\right]\\
			&\geq&p\left[ \frac{\lambda_2 e^{\lambda_2 t}}{e^{\lambda_2 t}-\overline{\alpha}}-\frac{\lambda_1 e^{\lambda_1 t}}{e^{\lambda_1 t}-\overline{\alpha}}\right]\left[p+q-1-(n-1) \right]\\
			&=&0.   
		\end{eqnarray*}
		Which implies that $M_2\leq M_3$. Let 
		\begin{eqnarray*}
			&&P=q\left[p  \frac{\lambda_1 e^{\lambda_1 t}}{e^{\lambda_1 t}-\overline{\alpha}}+(q-1)\frac{\lambda e^{\lambda t}}{e^{\lambda t}-\overline{\alpha}} \right]\left( \frac{e^{\lambda t}-\overline{\alpha}}{\alpha} \right)-(n-1)\left[ p\frac{\lambda_1 e^{\lambda_1 t}}{e^{\lambda_1 t}-\overline{\alpha}}+q\frac{\lambda e^{\lambda t}}{e^{\lambda t}-\overline{\alpha}}\right]\\
			&&Q= p\left[ (p-1)\frac{\lambda_1 e^{\lambda_1 t}}{e^{\lambda_1 t}-\overline{\alpha}}+q\frac{\lambda e^{\lambda t}}{e^{\lambda t}-\overline{\alpha}} \right]\\
			&&U=q\left( \frac{e^{\lambda t}-\overline{\alpha}}{\alpha} \right)-(n-1)  \quad and \quad V=p.
		\end{eqnarray*}
		So, we have
		\begin{eqnarray*}
			QU-PV&=&pq\left( \frac{e^{\lambda t}-\overline{\alpha}}{\alpha} \right)\left[  \frac{\lambda e^{\lambda t}}{e^{\lambda t}-\overline{\alpha}}-\frac{\lambda_1 e^{\lambda_1 t}}{e^{\lambda_1 t}-\overline{\alpha}}\right]+p(n-1)\left( \frac{\lambda_1 e^{\lambda_1 t}}{e^{\lambda_1 t}-\overline{\alpha}}\right)\geq0.  
		\end{eqnarray*}
		So,we have 
		\begin{eqnarray*}
			M_2-M_1&=&\frac{Q\left( \frac{e^{\lambda_2t}-\overline{\alpha}}{\alpha}\right)+P}{V\left( \frac{e^{\lambda_2t}-\overline{\alpha}}{\alpha}\right)+U}-\frac{Q\left( \frac{e^{\lambda_1t}-\overline{\alpha}}{\alpha}\right)+P}{V\left( \frac{e^{\lambda_1t}-\overline{\alpha}}{\alpha}\right)+U}\\
			&=&\frac{\left[Q\left( \frac{e^{\lambda_2t}-\overline{\alpha}}{\alpha}\right)+P \right]\left[V\left( \frac{e^{\lambda_1t}-\overline{\alpha}}{\alpha}\right)+U \right]-\left[Q\left( \frac{e^{\lambda_1t}-\overline{\alpha}}{\alpha}\right)+P \right]\left[ V\left( \frac{e^{\lambda_2t}-\overline{\alpha}}{\alpha}\right)+U\right]    }{\left[  V\left( \frac{e^{\lambda_2t}-\overline{\alpha}}{\alpha}\right)+U\right] \left[  V\left( \frac{e^{\lambda_1t}-\overline{\alpha}}{\alpha}\right)+U\right] } \\
			&\overset{\rm sgn}{=}&(QU-PV)\left(  \frac{e^{\lambda_2t}-\overline{\alpha}}{\alpha}- \frac{e^{\lambda_1t}-\overline{\alpha}}{\alpha}\right)\\
			&\geq&0, 
		\end{eqnarray*}
		which means that $M_1\leq M_2$. Thus, $M_1\leq M_3$, and the theorem follows.
		
	\end{proof}

	In the next, we turn to discuss the effect generated by the discrepancy among sample sizes on the hazard rate function of the second-order statistic arising from multiple-outlier MPHR samples.
	\begin{theorem}\label{thrf}
		Let $X_{1}, X_{2}, \dots,X_{p},X_{p+1},\dots,X_{n}$ be independentrandom variables following the multiple-outlier MPHR moder with survival functions$(\frac{\alpha\overline{F}^{\lambda_1}(x)}{1-\overline{\alpha}\overline{F}^{\lambda_1}(x)}{\rm I}_p,\frac{\alpha\overline{F}^{\lambda_2}(x)}{1-\overline{\alpha}\overline{F}^{\lambda_2}(x)}{\rm I}_q),$ where $p+q=n,p,q\geq1$ and  $\overline{F}(x)$ is the baseline survival function.Let $Y_{1}, Y_{2}, \dots,Y_{p^*},Y_{p^*+1},\dots,Y_{n^*}$ be another set of independent random variables following the multiple-outlier MPHR moder with survival functions$(\frac{\alpha\overline{F}^{\lambda_1}(x)}{1-\overline{\alpha}\overline{F}^{\lambda_1}(x)}{\rm I}_p^*,\frac{\alpha\overline{F}^{\lambda_2}(x)}{1-\overline{\alpha}\overline{F}^{\lambda_2}(x)}{\rm I}_q^*),$ where $p^*+q^*=n^*,p^*,q^*\geq1$. Denote by $X_{2:n}$ and $Y_{2:n^*}$ the second-order statistic arising from these two sets of multiple-outlier MPHR models, respectively. Suppose that $p^*\leq p\leq q\leq q^*$ and $\lambda_1\leq\lambda_2$.Then, we have 
		\begin{equation*}
			(p,q)\preceq_{\rm w}(p^*,q^*)\Rightarrow X_{2:n}\geq_{\rm hr} Y_{2:n^*}
		\end{equation*}

	\end{theorem}
	
	\begin{proof}
		The idea of the proof is borrowed from Theorem 3.4 of \cite{cai2017hazard}.
		Following the notation in the proof of Theorem 4, the hazard rate function of $X_{2:n}$ can be written as follows:
		\begin{eqnarray*}
			r_{X_{2:n}}(t)=\frac{pT_1\left( \frac{e^{\lambda_1t}-\overline{\alpha}}{\alpha}\right) +qT_2\left( \frac{e^{\lambda_2 t}-\overline{\alpha}}{\alpha}\right) -(n-1)T}{p\left( \frac{e^{\lambda_1t}-\overline{\alpha}}{\alpha}\right) +q\left( \frac{e^{\lambda_2 t}-\overline{\alpha}}{\alpha}\right) -(n-1)},
		\end{eqnarray*}
		where  $T_1=(p-1)\frac{\lambda_1e^{\lambda_1t}}{e^{\lambda_1t}-\overline{\alpha}}+q\frac{\lambda_{2} e^{\lambda_2 t}}{e^{\lambda_2 t}-\overline{\alpha}}$, $T_2=p\frac{\lambda_1e^{\lambda_1t}}{e^{\lambda_1t}-\overline{\alpha}}+(q-1)\frac{\lambda_2 e^{\lambda_2 t}}{e^{\lambda_2 t}-\overline{\alpha}}$ and $T=p\frac{\lambda_1e^{\lambda_1t}}{e^{\lambda_1t}-\overline{\alpha}}+q\frac{\lambda_2 e^{\lambda_2 t}}{e^{\lambda_2 t}-\overline{\alpha}}$. Denote by $r_{Y_{2:n^*}}(t)$ the hazard rate function of $Y_{2:n^*}$. We need to show that $r_{X_{2:n}}(t)\leq r_{Y_{2:n^*}}(t)$, i.e.,
		\begin{eqnarray*}
			\frac{pT_1\left( \frac{e^{\lambda_1t}-\overline{\alpha}}{\alpha}\right) +qT_2\left( \frac{e^{\lambda_2 t}-\overline{\alpha}}{\alpha}\right) -(n-1)T}{p\left( \frac{e^{\lambda_1t}-\overline{\alpha}}{\alpha}\right) +q\left( \frac{e^{\lambda_2 t}-\overline{\alpha}}{\alpha}\right) -(n-1)}
			\leq\frac{p^*T^*_1\left( \frac{e^{\lambda_1t}-\overline{\alpha}}{\alpha}\right) +q^*T^*_2\left( \frac{e^{\lambda_2 t}-\overline{\alpha}}{\alpha}\right) -(n^*-1)T^*}{p^*\left( \frac{e^{\lambda_1t}-\overline{\alpha}}{\alpha}\right) +q^*\left( \frac{e^{\lambda_2 t}-\overline{\alpha}}{\alpha}\right) -(n^*-1)}.
		\end{eqnarray*}
		where $T^*_1=(p^*-1)\frac{\lambda_1e^{\lambda_1t}}{e^{\lambda_1t}-\overline{\alpha}}+q^*\frac{\lambda_{2} e^{\lambda_2 t}}{e^{\lambda_2 t}-\overline{\alpha}}$, $T^*_2=p^*\frac{\lambda_1e^{\lambda_1t}}{e^{\lambda_1t}-\overline{\alpha}}+(q^*-1)\frac{\lambda_2 e^{\lambda_2 t}}{e^{\lambda_2 t}-\overline{\alpha}}$ and $T^*=p^*\frac{\lambda_1e^{\lambda_1t}}{e^{\lambda_1t}-\overline{\alpha}}+q^*\frac{\lambda_2 e^{\lambda_2 t}}{e^{\lambda_2 t}-\overline{\alpha}}$. Let $a_i=\frac{\lambda_ie^{\lambda_it}}{e^{\lambda_it}-\overline{\alpha}}$, $b_i=\left( \frac{e^{\lambda_it}-\overline{\alpha}}{\alpha}\right)$ and $c_{ij}=a_ib_j$ for $i,j=1,2.$ Denote
		\begin{eqnarray*}
			\phi(p,q)&=&\frac{pT_1\left( \frac{e^{\lambda_1t}-\overline{\alpha}}{\alpha}\right) +qT_2\left( \frac{e^{\lambda_2 t}-\overline{\alpha}}{\alpha}\right) -(n-1)T}{p\left( \frac{e^{\lambda_1t}-\overline{\alpha}}{\alpha}\right) +q\left( \frac{e^{\lambda_2 t}-\overline{\alpha}}{\alpha}\right) -(n-1)}\\
			&=&\frac{p(p-1)c_{11}+pqc_{21}+pqc_{12}+q(q-1)c_{22}-(n-1)(pa_1+qa_2)}{pb_1+qb_2-(n-1)}.
		\end{eqnarray*}
		It suffices to prove $\phi(p,q)\leq\phi(p^*,q^*)$ under the condition $p^*\leq p\leq q\leq q^*$ and $(p,q)\preceq_{\rm w}(p^*,q^*)$. Taking the derivative of $\phi(p,q)$ with respect to $p$, we have
		\begin{eqnarray*}
			\frac{\partial\phi(p,q)}{\partial p}&\overset{\rm sgn}{=}&\left[(2p-1)c_{11}+q(c_{21}+c_{12})-(n-1)a_1-pa_1-qa_2 \right]\times\left[ pb_1+qb_2-(n-1)\right]\\
			&&-\left[ p(p-1)c_{11}+pqc_{21}+pqc_{12}+q(q-1)c_{22}-(n-1)(pa_1+qa_2)\right]\times(b_1-1)\\
			&=&a_1\left[pb_1+qb_2-(n-1)\right]\times\left[ (p-1)b_1+qb_2-(n-1)\right]+(pc_{11}+qc_{22})(b_1-1) \\
			&\geq&0.   
		\end{eqnarray*}	
		Similarly,
		\begin{eqnarray*}
			\frac{\partial\phi(p,q)}{\partial q}&\overset{\rm sgn}{=}&a_2\left[pb_1+qb_2-(n-1)\right]\times\left[ pb_1+(q-1)b_2-(n-1)\right]+(pc_{11}+qc_{22})(b_2-1) \\
			&\geq&0.
		\end{eqnarray*}
		From the proof of Theorem 4, we get $a_2\geq a_1\geq0$ and $b_2\geq b_1\geq1$.Thus, we have
		\begin{eqnarray*}
			\frac{\partial\phi(p,q)}{\partial q}-\frac{\partial\phi(p,q)}{\partial p}&\overset{\rm sgn}{=}&\left[ pc_{21}+(q-1)c_{22}-(p-1)c_{11}-qc_{12}-(n-1)(a_2-a_1)\right]\\
			&&\times \left[pb_1+qb_2-(n-1) \right]+(pc_{11}+qc_{22})(b_2-b_1)\\
			&\geq&\left[ pc_{21}+(q-1)c_{22}-(p-1)c_{11}-qc_{12}-(n-1)(a_2-a_1)\right]\\
			&&\times \left[pb_1+qb_2-(n-1) \right]+(pb_1+qb_2)(a_1b_2-a_1b_1)\\
			&=&\left[ pc_{21}+(q-1)c_{22}-pc_{11}-(q-1)c_{12}-(n-1)(a_2-a_1)\right]\\
			&&\times \left[pb_1+qb_2-(n-1) \right]+(n-1)(c_{12}-c_{11})\\
			&\geq&(a_2-a_1)\left[ pb_1+(q-1)b_2-(n-1)\right]+(n-1)(c_{12}-c_{11})\\
			&\geq&0.  
		\end{eqnarray*}
		Now, the desired result follows from lemma \ref{wconvex}.

	\end{proof}
	\begin{remark}
		The results of theorem \ref{thrp} and theorem \ref{thrf} improve Theorem 3.1 and theorem 3.4 of \cite{cai2017hazard} from the PHR model to the case of the MPHR model. 
	\end{remark}
	
	The following numerical example is provided as an illustration of Theorem \ref{thrf}.
	\begin{example}
		Suppose $\overline{F}(x)=e^{-(ax)^b}, a>0, b>0$, Set $n=7, n^*=9, a=1.5, b=0.2, \alpha=0.05, \lambda_1=0.1, \lambda_2=0.3, p=3, q=4, p^*=1, q^*=8.$ It is easy to show that conditions $\lambda_1\leq\lambda_2, p^*\leq p\leq q\leq q^*$ and $(p,q)\preceq_{\rm w}(p^*,q^*)$ are all statisfied in theorem \ref{thrf}. Figure 4 shows the hazard rate functions of $X_{2:7}$ and $X_{2:9}$, from which it can be obversed that $r_{X_{2:7}}(x)$ is less than $r_{Y_{2:9}}(x)$, for $x>0$, thus validating the result in theorem \ref{thrf}.
		\begin{figure}[H]
			\centering
			%  \setlength{\abovecaptionskip}{-0.1cm}   
			%  \vspace{0.1cm}  
			%  \subfigcapskip=-8pt  
			
			{\includegraphics[width=7cm]{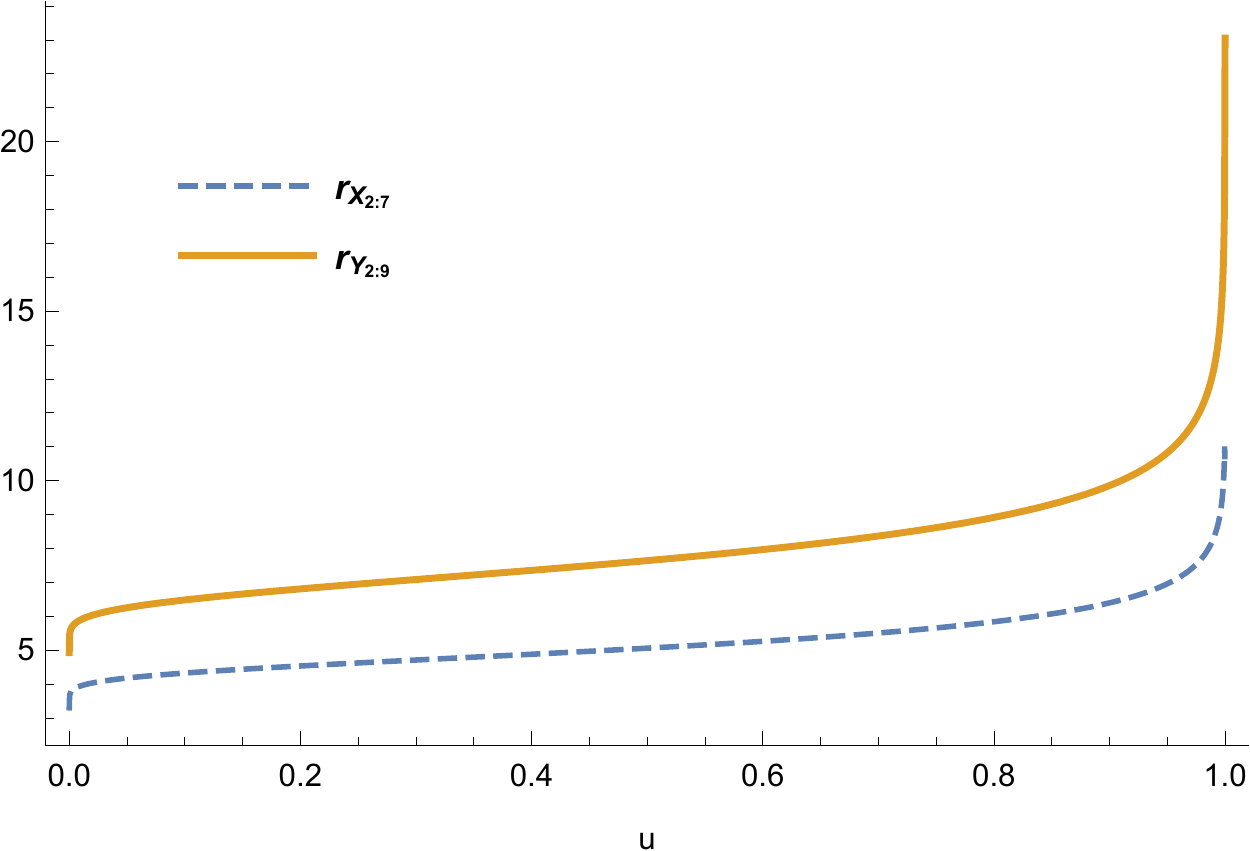}\label{fig2a}}
			
			\caption{Plots of  hazard rate functions of $X_{2:7}[Y_{2:9}]$, for all $x=-\ln u$ and  $u\in(0,1]$.
			}
		\end{figure}
		
	\end{example}

	\section{Concluding remarks}\label{con}
	
	In this paper, we study stochastic comparisons on the second-order statistics from heterogeneous dependent or independent MPHR samples. Some ordering results are established for the usual stochastic, hazard rate orderings on the second-order statistics. These new
	results established here provide theoretical guidance both for the winner’s prize for bid in the second-price reverse auction in auction theory and fail-safe system design in reliability theory.
	
	\section*{Funding}
	This research is supported by the National Natural Science Foundation of China.

	\bibliography{ref-quan}

\end{document}